%% file: GAU_MUL.tex
\documentclass[11pt]{article}
\usepackage[usenames]{xcolor}
\usepackage{enumerate}

\usepackage[square,numbers]{natbib}
\usepackage{undertilde}  
\usepackage{latexsym,amsmath,amssymb,amsbsy, amsthm}

\usepackage[left=1.2truein, right=1.2truein, top = 1.2truein, bottom = 1.2truein]{geometry}

\usepackage[CJKbookmarks=true,
            bookmarksnumbered=true,
			bookmarksopen=true,
			colorlinks=true,
			citecolor=red,
			linkcolor=blue,
			anchorcolor=red,
			urlcolor=blue]{hyperref}
\usepackage{setspace}
\usepackage{lscape,fancyhdr,fancybox}
\usepackage{graphicx,epsfig, xy}
\usepackage{float}
\usepackage{graphicx}
\usepackage{caption}
\usepackage{subcaption}

\newtheorem{cor}{Corollary}
\newcommand{\nb}[1]{\textcolor{blue}{\texttt{[#1]}}}
\newcommand{\ER}{{Erd\H{o}s--R\'{e}nyi}}

\newcommand{\bla}{\mathbf{u}}
\newcommand{\inn}[2]{
\langle #1 , #2 \rangle
}

\newcommand{\BSC}{{B}}
\begin{document}
\title{Asymptotic normality and analysis of variance of log-likelihood ratios in spiked random matrix models}
\author{Debapratim Banerjee and Zongming Ma \\
 ~\\
\textit{University of Pennsylvania}
}

\maketitle
 
\begin{abstract}
The present manuscript studies signal detection by likelihood ratio tests in a number of spiked random matrix models, including but not limited to Gaussian mixtures and spiked Wishart covariance matrices.
We work directly with multi-spiked cases in these models and with flexible priors on the signal component that allow dependence across spikes.
We derive asymptotic normality for the log-likelihood ratios when the signal-to-noise ratios are below certain thresholds.
In addition, we show that the variances of the log-likelihood ratios can be asymptotically decomposed as the sums of those of a collection of statistics which we call bipartite signed cycles.\\
~\\
\textit{Keywords:} Contiguity; Finite rank deformation; Principal Component Analysis; Random graphs; Signal detection.
\end{abstract}  

% Things to do
% 1. power of lss without fourth moment matching
% 2. detailed comparison with known results
% 3. tightness of the SNR condition 
% 4. argue what is new compared to Banerjee (16): mutliple spike, asymmetric data, subGaussian priors, the set of "sufficient statistics"

\input{defs.tex}

\input{intro.tex}

\input{results.tex}

\input{cycle.tex}
\input{result-proof.tex}

\appendix
\input{contiguity.tex}
\bibliographystyle{abbrv}
\bibliography{zm}

\end{document}

%% file: defs.tex
\newcommand{\wh}{\widehat}
\newcommand{\wt}{\widetilde}
\newcommand{\bem}{\begin{bmatrix}}
\newcommand{\eem}{\end{bmatrix}}

\newcommand{\hf}{1/2}
\newcommand{\eps}{\epsilon}
\newcommand{\vareps}{\varepsilon}

\newcommand{\eg}{e.g.\xspace}
\newcommand{\ie}{i.e.\xspace}
\newcommand{\iid}{i.i.d.\xspace}
\newcommand{\sfGamma}{{\mathsf{\Gamma}}}
\newcommand{\qsfGamma}{{\mathsf{\sqrt{\Gamma}}}}
\newcommand{\mmse}{{\mathsf{mmse}}}
\newcommand{\mse}{{\mathsf{mse}}}
\newcommand{\MMSE}{{\mathsf{MMSE}}}
\newcommand{\lmmse}{{\mathsf{lmmse}}}
\newcommand{\NGn}{N_{\sf G}}
\newcommand{\XGn}{X_{\sf G}}
\newcommand{\YGn}{Y_{\sf G}}
\newcommand{\DoF}{{\sf DoF}}
\newcommand{\dof}{{\sf dof}}

\newcommand{\ones}{\mathbf 1}
\newcommand{\reals}{{\mathbb{R}}}
\newcommand{\integers}{{\mathbb{Z}}}
\newcommand{\naturals}{{\mathbb{N}}}
\newcommand{\rationals}{{\mathbb{Q}}}
\newcommand{\naturalsex}{\overline{\mathbb{N}}}
\newcommand{\symm}{{\mbox{\bf S}}}  % symmetric matrices
\newcommand{\supp}{{\rm supp}}
\newcommand{\eexp}{{\rm e}}
\newcommand{\rexp}[1]{{\rm e}^{#1}}

\newcommand{\diff}{{\rm d}}
\newcommand{\deriv}[2]{{\frac{{\rm d} #1}{{\rm d} #2}}}
\newcommand{\Bigo}[1]{{O\left(#1\right)}}
\newcommand{\smallo}[1]{{o\left(#1\right)}}
\newcommand{\Sdelta}{{S_{\delta}^N}}
\newcommand{\deltaN}{{\left\lfloor \delta N \right\rfloor}}

\newcommand{\proj}{{\rm proj}}
\newcommand{\Null}{{\text{Null}}}
\newcommand{\Ker}{{\rm Ker}}
\newcommand{\Img}{{\rm Im}}
\newcommand{\Span}{{\rm span}}
\newcommand{\rank}{\mathop{\sf rank}}
\newcommand{\Tr}{\mathop{\sf Tr}}
\newcommand{\diag}{\mathop{\text{diag}}}
\newcommand{\lambdamax}{{\lambda_{\rm max}}}
\newcommand{\lambdamin}{\lambda_{\rm min}}
\newcommand{\E}{\mathrm{E}}
\newcommand{\Var}{\mathrm{Var}}
%% probability stuff
\newcommand{\CDF}{\mathsf{CDF}}
\newcommand{\Expect}{\mathbb{E}}
\newcommand{\expect}[1]{\mathbb{E} #1 }
\newcommand{\expects}[2]{\mathbb{E}_{#2}\left[ #1 \right]}
\newcommand{\tExpect}{{\tilde{\mathbb{E}}}}
\newcommand{\Prob}{\mathbb{P}}
\newcommand{\prob}[1]{{ \mathbb{P}\left\{ #1 \right\} }}
\newcommand{\tProb}{{\tilde{\mathbb{P}}}}
\newcommand{\tprob}[1]{{ \tProb\left\{ #1 \right\} }}
\newcommand{\hProb}{\hat{\mathbb{P}}}
\newcommand{\toprob}{\xrightarrow{\Prob}}
\newcommand{\tolp}[1]{\xrightarrow{L^{#1}}}
\newcommand{\toas}{\xrightarrow{{\rm a.s.}}}
\newcommand{\toae}{\xrightarrow{{\rm a.e.}}}
\newcommand{\todistr}{\xrightarrow{{\rm D}}}
\newcommand{\toweak}{\rightharpoonup}
\newcommand{\eqdistr}{{\stackrel{\rm (d)}{=}}}
\newcommand{\iiddistr}{{\stackrel{\text{\iid}}{\sim}}}
\newcommand{\Cov}{\text{Cov}}
\newcommand\indep{\protect\mathpalette{\protect\independenT}{\perp}}
\def\independenT#1#2{\mathrel{\rlap{$#1#2$}\mkern2mu{#1#2}}}
\newcommand{\Bern}{\text{Bern}}
\newcommand{\Binomial}{\text{Binomal}}

\newcommand{\Co}{{\mathop {\bf co}}} % convex hull
\newcommand{\dist}{\mathop{\bf dist{}}}
\newcommand{\argmin}{\mathop{\rm argmin}}
\newcommand{\argmax}{\mathop{\rm argmax}}
\newcommand{\epi}{\mathop{\bf epi}} % epigraph
\newcommand{\Vol}{\mathop{\bf vol}}
\newcommand{\dom}{\mathop{\bf dom}} % domain
\newcommand{\intr}{\mathop{\bf int}}
\newcommand{\closure}{\mathop{\bf cl}}

\newcommand{\trans}{^{\rm T}}
\newcommand{\Th}{{^{\rm th}}}
\newcommand{\diverge}{\to \infty}

%\newcommand{\fracp}[2]{\frac{\partial #1}{\partial #2}}
%\newcommand{\fracpk}[3]{\frac{\partial^{#3} #1}{\partial #2^{#3}}}
%\newcommand{\fracd}[2]{\frac{\diff #1}{\diff #2}}

%\ifthenelse{\boolean{slides}}{
\theoremstyle{remark}
\newtheorem{remark}{Remark}
\theoremstyle{plain}
\newtheorem{problem}{Problem}
\newtheorem{lemma}{Lemma}
\newtheorem{theorem}{Theorem}

\theoremstyle{definition}
\newtheorem{definition}{Definition}
\newtheorem{example}{Example}

\newtheorem{exer}{Exercise}
\newtheorem{soln}{Solution}
\newtheorem{condition}{Condition}
% }

\theoremstyle{plain}
\newtheorem{algo}{Algorithm}
\newtheorem{coro}{Corollary}
\newtheorem{prop}{Proposition}
\newtheorem{claim}{Claim}
\newtheorem{fprop}{False proposition}
\newtheorem{assu}{Assumption}

% for prettyref.sty
%\newrefformat{eq}{(\ref{#1})}
%\newrefformat{chap}{Chapter~\ref{#1}}
%\newrefformat{sec}{Section~\ref{#1}}
%\newrefformat{algo}{Algorithm~\ref{#1}}
%\newrefformat{fig}{Fig.~\ref{#1}}
%\newrefformat{tab}{Table~\ref{#1}}
%\newrefformat{rmk}{Remark~\ref{#1}}
%\newrefformat{clm}{Claim~\ref{#1}}
%\newrefformat{def}{Definition~\ref{#1}}
%\newrefformat{cor}{Corollary~\ref{#1}}
%\newrefformat{lmm}{Lemma~\ref{#1}}
%\newrefformat{lemma}{Lemma~\ref{#1}}
%\newrefformat{prop}{Proposition~\ref{#1}}
%\newrefformat{app}{Appendix~\ref{#1}}
%\newrefformat{ex}{Example~\ref{#1}}
%\newrefformat{exer}{Exercise~\ref{#1}}
%\newrefformat{soln}{Solution~\ref{#1}}
%\newrefformat{cond}{Condition~\ref{#1}}

\newenvironment{colortext}[1]{\color{#1}}{}

\newcommand{\lunder}[1]{{\underset{\raise0.3em\hbox{$\smash{\scriptscriptstyle-}$}}{#1}}}

% 11/12/2007
\newcommand{\floor}[1]{{\left\lfloor {#1} \right \rfloor}}
\newcommand{\ceil}[1]{{\left\lceil {#1} \right \rceil}}

% various metrics
% 11/27/2007: L^p spaces
% 05/04/2010: distances between probability measures
% 04/28/2011: matrix norms
\newcommand{\norm}[1]{\left\|{#1} \right\|}
\newcommand{\Norm}[1]{\|{#1} \|}
\newcommand{\lpnorm}[1]{\left\|{#1} \right\|_{p}}
\newcommand{\linf}[1]{\left\|{#1} \right\|_{\infty}}
\newcommand{\lnorm}[2]{\left\|{#1} \right\|_{{#2}}}
\newcommand{\Lploc}[1]{L^{#1}_{\rm loc}}
\newcommand{\hellinger}{d_{\rm H}}
\newcommand{\Fnorm}[1]{\lnorm{#1}{\rm F}}
\newcommand{\fnorm}[1]{\|#1\|_{\rm F}}
\newcommand{\Opnorm}[1]{\lnorm{#1}{\rm op}}
\newcommand{\opnorm}[1]{\|#1\|_{\rm op}}
\newcommand{\maxnorm}[1]{\|#1\|_{\max}}
% \newcommand{\opnorm}[1]{\left\| #1 \right\|}

% inner product
\newcommand{\iprod}[2]{\left \langle #1, #2 \right\rangle}
\newcommand{\Iprod}[2]{\langle #1, #2 \rangle}

% 12/02/2007
\newcommand{\indc}[1]{{\mathbf{1}_{\left\{{#1}\right\}}}}
\newcommand{\Indc}{\mathbf{1}}

% 12/17/2007, courtesy of LittleLeo@newsmth
\def\innergetnumber#1[#2]#3{#2}
\def\getnumber{\expandafter\innergetnumber\jobname}

% 2/25/2008, dimension stuff
\newcommand{\Rdim}{{\dim_{\rm R}}}
\newcommand{\uRdim}{{\overline{\dim}_{\rm R}}}
\newcommand{\lRdim}{{\underline{\dim}_{\rm R}}}
\newcommand{\Idim}{{\dim_{\rm I}}}
\newcommand{\uIdim}{{\overline{\dim}_{\rm I}}}
\newcommand{\lIdim}{{\underline{\dim}_{\rm I}}}
\newcommand{\Bdim}{{\dim_{\rm B}}}
\newcommand{\uBdim}{\overline{\dim}_{\rm B}}
\newcommand{\lBdim}{\underline{\dim}_{\rm B}}
\newcommand{\Hdim}{{\dim_{\rm H}}}
\newcommand{\inddim}{{\dim_{\rm ind}}}
\newcommand{\locdim}{{\dim_{\rm loc}}}
\newcommand{\od}{{\overline{d}}}
\newcommand{\ud}{{\underline{d}}}

%% symbols
%\newcommand{\bs}[1]{{\boldsymbol{#1}}}
\newcommand{\bszero}{{\boldsymbol{0}}}
\newcommand{\bsvphi}{{\boldsymbol{\varphi}}}
\newcommand{\bsPhi}{{\boldsymbol{\Phi}}}
\newcommand{\bsXi}{{\boldsymbol{\Xi}}}

\newcommand{\bsa}{{\boldsymbol{a}}}
\newcommand{\bsb}{{\boldsymbol{b}}}
\newcommand{\bsc}{{\boldsymbol{c}}}
\newcommand{\bsd}{{\boldsymbol{d}}}
\newcommand{\bse}{{\boldsymbol{e}}}
\newcommand{\bsf}{{\boldsymbol{f}}}
\newcommand{\bsg}{{\boldsymbol{g}}}
\newcommand{\bsh}{{\boldsymbol{h}}}
\newcommand{\bsi}{{\boldsymbol{i}}}
\newcommand{\bsj}{{\boldsymbol{j}}}
\newcommand{\bsk}{{\boldsymbol{k}}}
\newcommand{\bsl}{{\boldsymbol{l}}}
\newcommand{\bsm}{{\boldsymbol{m}}}
\newcommand{\bsn}{{\boldsymbol{n}}}
\newcommand{\bso}{{\boldsymbol{o}}}
\newcommand{\bsp}{{\boldsymbol{p}}}
\newcommand{\bsq}{{\boldsymbol{q}}}
\newcommand{\bsr}{{\boldsymbol{r}}}
\newcommand{\bss}{{\boldsymbol{s}}}
\newcommand{\bst}{{\boldsymbol{t}}}
\newcommand{\bsu}{{\boldsymbol{u}}}
\newcommand{\bsv}{{\boldsymbol{v}}}
\newcommand{\bsw}{{\boldsymbol{w}}}
\newcommand{\bsx}{{\boldsymbol{x}}}
\newcommand{\bsy}{{\boldsymbol{y}}}
\newcommand{\bsz}{{\boldsymbol{z}}}
\newcommand{\bsA}{{\boldsymbol{A}}}
\newcommand{\bsB}{{\boldsymbol{B}}}
\newcommand{\bsC}{{\boldsymbol{C}}}
\newcommand{\bsD}{{\boldsymbol{D}}}
\newcommand{\bsE}{{\boldsymbol{E}}}
\newcommand{\bsF}{{\boldsymbol{F}}}
\newcommand{\bsG}{{\boldsymbol{G}}}
\newcommand{\bsH}{{\boldsymbol{H}}}
\newcommand{\bsI}{{\boldsymbol{I}}}
\newcommand{\bsJ}{{\boldsymbol{J}}}
\newcommand{\bsK}{{\boldsymbol{K}}}
\newcommand{\bsL}{{\boldsymbol{L}}}
\newcommand{\bsM}{{\boldsymbol{M}}}
\newcommand{\bsN}{{\boldsymbol{N}}}
\newcommand{\bsO}{{\boldsymbol{O}}}
\newcommand{\bsP}{{\boldsymbol{P}}}
\newcommand{\bsQ}{{\boldsymbol{Q}}}
\newcommand{\bsR}{{\boldsymbol{R}}}
\newcommand{\bsS}{{\boldsymbol{S}}}
\newcommand{\bsT}{{\boldsymbol{T}}}
\newcommand{\bsU}{{\boldsymbol{U}}}
\newcommand{\bsV}{{\boldsymbol{V}}}
\newcommand{\bsW}{{\boldsymbol{W}}}
\newcommand{\bsX}{{\boldsymbol{X}}}
\newcommand{\bsY}{{\boldsymbol{Y}}}
\newcommand{\bsZ}{{\boldsymbol{Z}}}

\newcommand{\bbA}{{\mathbb{A}}}
\newcommand{\bbB}{{\mathbb{B}}}
\newcommand{\bbC}{{\mathbb{C}}}
\newcommand{\bbD}{{\mathbb{D}}}
\newcommand{\bbE}{{\mathbb{E}}}
\newcommand{\bbF}{{\mathbb{F}}}
\newcommand{\bbG}{{\mathbb{G}}}
\newcommand{\bbH}{{\mathbb{H}}}
\newcommand{\bbI}{{\mathbb{I}}}
\newcommand{\bbJ}{{\mathbb{J}}}
\newcommand{\bbK}{{\mathbb{K}}}
\newcommand{\bbL}{{\mathbb{L}}}
\newcommand{\bbM}{{\mathbb{M}}}
\newcommand{\bbN}{{\mathbb{N}}}
\newcommand{\bbO}{{\mathbb{O}}}
\newcommand{\bbP}{{\mathbb{P}}}
\newcommand{\bbQ}{{\mathbb{Q}}}
\newcommand{\bbR}{{\mathbb{R}}}
\newcommand{\bbS}{{\mathbb{S}}}
\newcommand{\bbT}{{\mathbb{T}}}
\newcommand{\bbU}{{\mathbb{U}}}
\newcommand{\bbV}{{\mathbb{V}}}
\newcommand{\bbW}{{\mathbb{W}}}
\newcommand{\bbX}{{\mathbb{X}}}
\newcommand{\bbY}{{\mathbb{Y}}}
\newcommand{\bbZ}{{\mathbb{Z}}}

\newcommand{\sfa}{{\mathsf{a}}}
\newcommand{\sfc}{{\mathsf{c}}}
\newcommand{\sfd}{{\mathsf{d}}}
\newcommand{\sfe}{{\mathsf{e}}}
\newcommand{\sff}{{\mathsf{f}}}
\newcommand{\sfg}{{\mathsf{g}}}
\newcommand{\sfh}{{\mathsf{h}}}
\newcommand{\sfi}{{\mathsf{i}}}
\newcommand{\sfj}{{\mathsf{j}}}
\newcommand{\sfk}{{\mathsf{k}}}
\newcommand{\sfl}{{\mathsf{l}}}
\newcommand{\sfm}{{\mathsf{m}}}
\newcommand{\sfn}{{\mathsf{n}}}
\newcommand{\sfo}{{\mathsf{o}}}
\newcommand{\sfp}{{\mathsf{p}}}
\newcommand{\sfq}{{\mathsf{q}}}
\newcommand{\sfr}{{\mathsf{r}}}
\newcommand{\sfs}{{\mathsf{s}}}
\newcommand{\sft}{{\mathsf{t}}}
\newcommand{\sfu}{{\mathsf{u}}}
\newcommand{\sfv}{{\mathsf{v}}}
\newcommand{\sfw}{{\mathsf{w}}}
\newcommand{\sfx}{{\mathsf{x}}}
\newcommand{\sfy}{{\mathsf{y}}}
\newcommand{\sfz}{{\mathsf{z}}}
\newcommand{\sfA}{{\mathsf{A}}}
\newcommand{\sfB}{{\mathsf{B}}}
\newcommand{\sfC}{{\mathsf{C}}}
\newcommand{\sfD}{{\mathsf{D}}}
\newcommand{\sfE}{{\mathsf{E}}}
\newcommand{\sfF}{{\mathsf{F}}}
\newcommand{\sfG}{{\mathsf{G}}}
\newcommand{\sfH}{{\mathsf{H}}}
\newcommand{\sfI}{{\mathsf{I}}}
\newcommand{\sfJ}{{\mathsf{J}}}
\newcommand{\sfK}{{\mathsf{K}}}
\newcommand{\sfL}{{\mathsf{L}}}
\newcommand{\sfM}{{\mathsf{M}}}
\newcommand{\sfN}{{\mathsf{N}}}
\newcommand{\sfO}{{\mathsf{O}}}
\newcommand{\sfP}{{\mathsf{P}}}
\newcommand{\sfQ}{{\mathsf{Q}}}
\newcommand{\sfR}{{\mathsf{R}}}
\newcommand{\sfS}{{\mathsf{S}}}
\newcommand{\sfT}{{\mathsf{T}}}
\newcommand{\sfU}{{\mathsf{U}}}
\newcommand{\sfV}{{\mathsf{V}}}
\newcommand{\sfW}{{\mathsf{W}}}
\newcommand{\sfX}{{\mathsf{X}}}
\newcommand{\sfY}{{\mathsf{Y}}}
\newcommand{\sfZ}{{\mathsf{Z}}}

\newcommand{\calA}{{\mathcal{A}}}
\newcommand{\calB}{{\mathcal{B}}}
\newcommand{\calC}{{\mathcal{C}}}
\newcommand{\calD}{{\mathcal{D}}}
\newcommand{\calE}{{\mathcal{E}}}
\newcommand{\calF}{{\mathcal{F}}}
\newcommand{\calG}{{\mathcal{G}}}
\newcommand{\calH}{{\mathcal{H}}}
\newcommand{\calI}{{\mathcal{I}}}
\newcommand{\calJ}{{\mathcal{J}}}
\newcommand{\calK}{{\mathcal{K}}}
\newcommand{\calL}{{\mathcal{L}}}
\newcommand{\calM}{{\mathcal{M}}}
\newcommand{\calN}{{\mathcal{N}}}
\newcommand{\calO}{{\mathcal{O}}}
\newcommand{\calP}{{\mathcal{P}}}
\newcommand{\calQ}{{\mathcal{Q}}}
\newcommand{\calR}{{\mathcal{R}}}
\newcommand{\calS}{{\mathcal{S}}}
\newcommand{\calT}{{\mathcal{T}}}
\newcommand{\calU}{{\mathcal{U}}}
\newcommand{\calV}{{\mathcal{V}}}
\newcommand{\calW}{{\mathcal{W}}}
\newcommand{\calX}{{\mathcal{X}}}
\newcommand{\calY}{{\mathcal{Y}}}
\newcommand{\calZ}{{\mathcal{Z}}}

\newcommand{\frakd}{{\mathscr{d}}}
\newcommand{\ofrakd}{{\overline{\frakd}}}
\newcommand{\ufrakd}{{\underline{\frakd}}}

\newcommand{\scrD}{{\mathscr{D}}}
\newcommand{\oscrD}{{\overline{\scrD}}}
\newcommand{\uscrD}{{\underline{\scrD}}}
\newcommand{\scrJ}{{\mathscr{J}}}
\newcommand{\oscrJ}{{\overline{\scrJ}}}
\newcommand{\uscrJ}{{\underline{\scrJ}}}

\newcommand{\bfa}{{\mathbf{a}}}
\newcommand{\bfb}{{\mathbf{b}}}
\newcommand{\bfc}{{\mathbf{c}}}
\newcommand{\bfd}{{\mathbf{d}}}
\newcommand{\bfe}{{\mathbf{e}}}
\newcommand{\bff}{{\mathbf{f}}}
\newcommand{\bfg}{{\mathbf{g}}}
\newcommand{\bfh}{{\mathbf{h}}}
\newcommand{\bfi}{{\mathbf{i}}}
\newcommand{\bfj}{{\mathbf{j}}}
\newcommand{\bfk}{{\mathbf{k}}}
\newcommand{\bfl}{{\mathbf{l}}}
\newcommand{\bfm}{{\mathbf{m}}}
\newcommand{\bfn}{{\mathbf{n}}}
\newcommand{\bfo}{{\mathbf{o}}}
\newcommand{\bfp}{{\mathbf{p}}}
\newcommand{\bfq}{{\mathbf{q}}}
\newcommand{\bfr}{{\mathbf{r}}}
\newcommand{\bfs}{{\mathbf{s}}}
\newcommand{\bft}{{\mathbf{t}}}
\newcommand{\bfu}{{\mathbf{u}}}
\newcommand{\bfv}{{\mathbf{v}}}
\newcommand{\bfw}{{\mathbf{w}}}
\newcommand{\bfx}{{\mathbf{x}}}
\newcommand{\bfy}{{\mathbf{y}}}
\newcommand{\bfz}{{\mathbf{z}}}
\newcommand{\bfA}{{\mathbf{A}}}
\newcommand{\bfB}{{\mathbf{B}}}
\newcommand{\bfC}{{\mathbf{C}}}
\newcommand{\bfD}{{\mathbf{D}}}
\newcommand{\bfE}{{\mathbf{E}}}
\newcommand{\bfF}{{\mathbf{F}}}
\newcommand{\bfG}{{\mathbf{G}}}
\newcommand{\bfH}{{\mathbf{H}}}
\newcommand{\bfI}{{\mathbf{I}}}
\newcommand{\bfJ}{{\mathbf{J}}}
\newcommand{\bfK}{{\mathbf{K}}}
\newcommand{\bfL}{{\mathbf{L}}}
\newcommand{\bfM}{{\mathbf{M}}}
\newcommand{\bfN}{{\mathbf{N}}}
\newcommand{\bfO}{{\mathbf{O}}}
\newcommand{\bfP}{{\mathbf{P}}}
\newcommand{\bfQ}{{\mathbf{Q}}}
\newcommand{\bfR}{{\mathbf{R}}}
\newcommand{\bfS}{{\mathbf{S}}}
\newcommand{\bfT}{{\mathbf{T}}}
\newcommand{\bfU}{{\mathbf{U}}}
\newcommand{\bfV}{{\mathbf{V}}}
\newcommand{\bfW}{{\mathbf{W}}}
\newcommand{\bfX}{{\mathbf{X}}}
\newcommand{\bfY}{{\mathbf{Y}}}
\newcommand{\bfZ}{{\mathbf{Z}}}

\newcommand{\bara}{{\bar{a}}}
\newcommand{\barb}{{\bar{b}}}
\newcommand{\barc}{{\bar{c}}}
\newcommand{\bard}{{\bar{d}}}
\newcommand{\bare}{{\bar{e}}}
\newcommand{\barf}{{\bar{f}}}
\newcommand{\barg}{{\bar{g}}}
\newcommand{\barh}{{\bar{h}}}
\newcommand{\bari}{{\bar{i}}}
\newcommand{\barj}{{\bar{j}}}
\newcommand{\bark}{{\bar{k}}}
\newcommand{\barl}{{\bar{l}}}
\newcommand{\barm}{{\bar{m}}}
\newcommand{\barn}{{\bar{n}}}
\newcommand{\barp}{{\bar{p}}}
\newcommand{\barq}{{\bar{q}}}
\newcommand{\barr}{{\bar{r}}}
\newcommand{\bars}{{\bar{s}}}
\newcommand{\bart}{{\bar{t}}}
\newcommand{\baru}{{\bar{u}}}
\newcommand{\barv}{{\bar{v}}}
\newcommand{\barw}{{\bar{w}}}
\newcommand{\barx}{{\bar{x}}}
\newcommand{\bary}{{\bar{y}}}
\newcommand{\barz}{{\bar{z}}}
\newcommand{\barA}{{\bar{A}}}
\newcommand{\barB}{{\bar{B}}}
\newcommand{\barC}{{\bar{C}}}
\newcommand{\barD}{{\bar{D}}}
\newcommand{\barE}{{\bar{E}}}
\newcommand{\barF}{{\bar{F}}}
\newcommand{\barG}{{\bar{G}}}
\newcommand{\barH}{{\bar{H}}}
\newcommand{\barI}{{\bar{I}}}
\newcommand{\barJ}{{\bar{J}}}
\newcommand{\barK}{{\bar{K}}}
\newcommand{\barL}{{\bar{L}}}
\newcommand{\barM}{{\bar{M}}}
\newcommand{\barN}{{\bar{N}}}
\newcommand{\barO}{{\bar{O}}}
\newcommand{\barP}{{\bar{P}}}
\newcommand{\barQ}{{\bar{Q}}}
\newcommand{\barR}{{\bar{R}}}
\newcommand{\barS}{{\bar{S}}}
\newcommand{\barT}{{\bar{T}}}
\newcommand{\barU}{{\bar{U}}}
\newcommand{\barV}{{\bar{V}}}
\newcommand{\barW}{{\bar{W}}}
\newcommand{\barX}{{\bar{X}}}
\newcommand{\barY}{{\bar{Y}}}
\newcommand{\barZ}{{\bar{Z}}}

\newcommand{\tmu}{{\tilde{\mu}}}
\newcommand{\tf}{{\tilde{f}}}
\newcommand{\tilh}{{\tilde{h}}}
\newcommand{\tu}{{\tilde{u}}}
\newcommand{\tx}{{\tilde{x}}}
\newcommand{\ty}{{\tilde{y}}}
\newcommand{\tz}{{\tilde{z}}}
\newcommand{\tA}{{\tilde{A}}}
\newcommand{\tB}{{\tilde{B}}}
\newcommand{\tC}{{\tilde{C}}}
\newcommand{\tD}{{\tilde{D}}}
\newcommand{\tE}{{\tilde{E}}}
\newcommand{\tF}{{\tilde{F}}}
\newcommand{\tG}{{\tilde{G}}}
\newcommand{\tH}{{\tilde{H}}}
\newcommand{\tI}{{\tilde{I}}}
\newcommand{\tJ}{{\tilde{J}}}
\newcommand{\tK}{{\tilde{K}}}
\newcommand{\tL}{{\tilde{L}}}
\newcommand{\tM}{{\tilde{M}}}
\newcommand{\tN}{{\tilde{N}}}
\newcommand{\tO}{{\tilde{O}}}
\newcommand{\tP}{{\tilde{P}}}
\newcommand{\tQ}{{\tilde{Q}}}
\newcommand{\tR}{{\tilde{R}}}
\newcommand{\tS}{{\tilde{S}}}
\newcommand{\tT}{{\tilde{T}}}
\newcommand{\tU}{{\tilde{U}}}
\newcommand{\tV}{{\tilde{\V}}}
\newcommand{\tW}{{\tilde{W}}}
\newcommand{\tX}{{\tilde{X}}}
\newcommand{\tY}{{\tilde{Y}}}
\newcommand{\tZ}{{\tilde{Z}}}

%% for ELE486
\newcommand{\binum}{{\{0,1\}}}
\newcommand{\spt}[1]{{\rm spt}(#1)}
\newcommand{\comp}[1]{{#1^{\rm c}}}
\newcommand{\Leb}{{\rm Leb}}

%% 03022008
\newcommand{\quant}[1]{{\left\langle #1\right\rangle}}
\newcommand{\sgn}{{\rm sgn}}

%% 03312008
\newcommand{\ntok}[2]{{#1,\ldots,#2}}

%% topo terminology
\newcommand{\uc}{{uniformly continuous}}
\newcommand{\holderc}{{H\"older continuous}}
\newcommand{\diam}{{\rm diam}}
\newcommand{\cardi}[1]{{\mathbf{card}(#1)}}
\newcommand{\restrict}[2]{\left.#1\right|_{#2}}

%% analysis
\newcommand{\Lip}{\mathrm{Lip}}
%\newcommand{\diverge}{\to \infty}

%% group
\newcommand{\id}{{\mathsf{id}}}
\newcommand{\GL}{{\mathrm{GL}}}

%% people
\newcommand{\renyi}{R\'enyi\xspace}
\newcommand{\Holder}{H\"older\xspace}
\newcommand{\levy}{L\'evy\xspace}

%% rates
%\newcommand{\rstable}{{r_{\rm stable}}}
%\newcommand{\rlip}{{r_{\rm Lip}}}
%\newcommand{\rucc}{{r_{\rm ucc}}}
%\newcommand{\rlinear}{{r_{\rm linear}}}
%\newcommand{\rB}{{r_{\rm B}}}
\newcommand{\rstable}{{\overline{\sfR}}}
\newcommand{\rlip}{{\sfR}}
\newcommand{\rll}{{\hat{\sfR}}}
\newcommand{\ruc}{{\sfR_0}}
\newcommand{\rucc}{{\sfR_0}}
\newcommand{\rlinear}{{\sfR^*}}
\newcommand{\rlinearsw}{{\sfR^*_{\rm SW}}}
\newcommand{\rB}{{\sfR_{\sf B}}}

%% distortion & sensitivity
\newcommand{\Dstar}{D^*}
\newcommand{\DLstar}{D^*_{\rm L}}
\newcommand{\DL}{D_{\rm L}}
\newcommand{\Rstar}{\calR^*}
\newcommand{\RLstar}{\calR^*_{\rm L}}
\newcommand{\RL}{\calR_{\rm L}}
\newcommand{\zstar}{\zeta^*}
\newcommand{\zLstar}{\zeta^*_{\rm L}}
\newcommand{\zL}{\zeta_{\rm L}}
\newcommand{\xstar}{\xi^*}
\newcommand{\xLstar}{\xi^*_{\rm L}}
\newcommand{\xL}{\xi_{\rm L}}

%% structural stuff
\newcommand{\ssec}[1]{\subsection{#1}}
\newcommand{\sssec}[1]{\subsubsection{#1}}
\newcommand{\separator}{\noindent \rule[0pt]{\textwidth}{2pt}}

%% parenthesis
\newcommand{\pth}[1]{\left( #1 \right)}
\newcommand{\qth}[1]{\left[ #1 \right]}
\newcommand{\sth}[1]{\left\{ #1 \right\}}
\newcommand{\bpth}[1]{\Bigg( #1 \Bigg)} % no auto sizing
\newcommand{\bqth}[1]{\Bigg[ #1 \Bigg]}
\newcommand{\bsth}[1]{\Bigg\{ #1 \Bigg\}}
\newcommand{\bfloor}[1]{\Bigg\lfloor #1 \Bigg\rfloor}

%% fracs
\newcommand{\fracp}[2]{\frac{\partial #1}{\partial #2}}
\newcommand{\fracpk}[3]{\frac{\partial^{#3} #1}{\partial #2^{#3}}}
\newcommand{\fracd}[2]{\frac{\diff #1}{\diff #2}}
\newcommand{\fracdk}[3]{\frac{\diff^{#3} #1}{\diff #2^{#3}}}

%% pushforward/pullback
\newcommand{\pushf}[2]{#1_{\sharp} #2}

% distances
\newcommand{\KL}[2]{D(#1 \, || \, #2)}
\newcommand{\Fisher}[2]{I(#1 \, || \, #2)}
\newcommand{\TV}{{\sf TV}}
\newcommand{\LC}{{\sf LC}}

%% coding theory stuff
\newcommand{\dham}{d_{\sf H}}
\newcommand{\ellmax}{\ell_{\max}}

%% communication
\newcommand{\bits}{\texttt{bits}\xspace}
\newcommand{\nats}{{\texttt{nats}}}
\newcommand{\snr}{{\mathsf{snr}}}
\newcommand{\qsnr}{{\sqrt{\mathsf{snr}}}}
\newcommand{\SNR}{{\mathsf{SNR}}}
\newcommand{\INR}{{\mathsf{INR}}}
\newcommand{\SIR}{{\mathsf{SIR}}}

%% todonotes
\newcommand{\todolong}[1]{\todo[inline,color=red!40]{#1}}

%% markups
\newcommand{\xxx}{\textbf{xxx}\xspace}

%% file: intro.tex
% !TEX root = GAU_MUL.tex

\section{Introduction} 

An important class of signal detection problems share the following hypothesis testing framework.
Under the null hypothesis, the observed data matrix consists of pure noise.
Under the alternative, its has a ``signal $+$ noise'' structure, where the signal component is of low rank and certain knowledge can be encoded as a prior distribution on the signal.

% \nb{general intro of bipartite and symmetric cases}
We consider two different cases of the problem. 
% namely the bipartite case and the symmetric case:
\begin{itemize}
\item {In the unnormalized case, we assume that the observed data matrix $X$ is in $\reals^{n\times p}$. Let $Z = (Z_{ij})\in \reals^{n\times p}$ with $Z_{ij} \stackrel{iid}{\sim} N(0,1)$. 
We aim to test 
\begin{equation}
\label{eq:bipartite}
H_0: X = Z, \quad \mbox{vs.} \quad
H_1: X = \frac{1}{\sqrt{p}} \Theta U' + Z,
\end{equation}
where $\Theta \in \reals^{n\times \kappa}$ follows some prior distribution $\pi_{\Theta}$ and $U \in \reals^{p\times\kappa}$ follows some prior distribution $\pi_{U}$ and $\Theta, U$ and $Z$ are mutually independent.
Throughout, we assume that under $\pi_\Theta$ (and $\pi_U$, resp.), the rows of $\Theta$ ($U$, resp.) are i.i.d.~random vectors with $\Expect[\Theta_{1*}]=0$ and $\Cov(\Theta_{1*}'\Theta_{1*}) = \Sigma_\Theta$ (with $\Expect[U_{1*}]=0$ and $\Cov(U_{1*}'U_{1*}) = \Sigma_U$, resp.).
In other words, we assume that $X_{ij} \stackrel{iid}{\sim} N(0,1)$ under $H_0$, and under $H_1$, $X_{ij}|(\Theta, U) \stackrel{ind}{\sim} N(\frac{1}{\sqrt{p}}\sum_{l=1}^\kappa \Theta_{i l}U_{jl},1)$ where $\Expect[\Theta_{il}] = 0$, $\Expect[U_{jl}] = 0$, $\Expect[\Theta_{il_1}\Theta_{il_2}] = \Sigma_\Theta(l_1,l_2)$ and
$\Expect[U_{j l_1} U_{j l_2}] = \Sigma_U(l_1,l_2)$ for $l,l_1,l_2 \in [\kappa]$.
Here and after, for any matrix $A\in\reals^{n\times p}$, $A'$ stands for its transposition, $A_{i*}\in \reals^{1\times p}$ its $i$-th row, and $A_{*j}\in \reals^{n\times 1}$ its $j$-th column. 
We use $A_{ij}$ and $A(i,j)$ exchangeably to denote its $(i,j)$-th entry.
For any positive integer $l$, $[l] = \{1,2,\dots, l\}$.
}	
% \nb{to discuss its connection to the spiked covariance model}
Under $H_1$, if the distribution of $\Theta_{i*}$ is discrete and takes a finite number of values, then conditioning on $U$ the rows of $X$ in \eqref{eq:bipartite} can be viewed as i.i.d.~observations from a Gaussian mixture distribution.
% \nb{change all $E$ to $Z$ for noise notation!!}

\item In the normalized case, we test
 \begin{equation}
\label{eq:bipartitevar}
H_0: X = Z, \quad \mbox{vs.} \quad
H_1: X = \Theta V' + Z,
\end{equation}
where $V= U(U'U)^{-{1}/{2}}$ and $U$ is defined as is in the previous case. 
In other words, $V$ is a self-normalized version of $U$ such that $V'V = I_\kappa$ and so $V\in O(p,\kappa)$, i.e., the Stiefel manifold consisting of all $\kappa$-frames in $\reals^p$. 
Under $H_1$, if $\Theta_{i*}\stackrel{iid}{\sim} N(0,\Sigma_\Theta)$ with $\Sigma_\Theta = H = \mathrm{diag}(h_1,\dots,h_\kappa)$ with $h_1\geq \cdots\geq h_\kappa > 0$, then conditioning on $V$ the rows of $X$ are i.i.d.~observations from a $p$-dimensional normal distribution with mean $0$ and multi-spiked covariance matrix $V H V' + I_p$. 
Here $I_p$ is the $p$-dimensional identity matrix.
In this case, \eqref{eq:bipartitevar} reduces to the high dimensional sphericity test against multi-spiked alternative.
\end{itemize}
% Note that given a matrix $X$, it can always be viewed as the adjacency matrix of a weighted graph.
% When it is asymmetric, the corresponding graph is bipartite and hence the name of the first case.
% Otherwise, the corresponding graph is symmetric and hence the name of the second case.
In either case, we aim to detect a spiked random matrix model against the null. 
Moreover, 
we deal with simple vs.~simple hypothesis testing since we put some prior distribution on the signal component under the alternative.
Therefore, the Neyman--Pearson lemma dictates that the likelihood ratio test is optimal.

% \nb{contiguous, singular, asymptotic distribution of likelihood ratio}
In this manuscript, we are concerned with the asymptotic behavior of likelihood ratios in the aforementioned testing problems.
In particular, let $p=p_n$ scale with $n$. 
We are interested in the asymptotic regime where 
\begin{equation}
\label{eq:asymp}
\mbox{$p/n\to \gamma \in (0,\infty)$ as $n\to\infty$.}
\end{equation}
% In the symmetric case, we are interested in the regime where $n\to\infty$.
Let $\bbP_{0,n}$ be the null distribution and $\bbP_{1,n}$ the alternative.
Let $L_n = \frac{\diff \bbP_{1,n}}{\diff \bbP_{0,n}}$ denote the likelihood ratio, and we call $\log(L_n)$ the log-likelihood ratio.

A series of papers have discovered the following general phenomenon in these testing problems.
Depending on the signal-to-noise ratio (SNR), there are two different types of asymptotic behavior of the likelihood ratio.
If the SNR is below certain threshold, then $L_n$ has a nontrivial weak limit, and the null and alternative distributions are mutually contiguous.
When the SNR is sufficiently large, the likelihood ratio converges to zero under null and diverges to infinity under alternative in probability as $n$ tends to infinity. 
In this case, the two distributions are asymptotically singular.
\citet{Banks17} focused on locating the boundary between asymptotically contiguous and singular regimes in both model \eqref{eq:bipartite} and its symmetric counterpart known as the Gaussian spiked Wigner model.
% , including the multi-spiked case.
% \nb{to be made more precise}
\citet{perry2016optimality} investigated the same issue in three single-spiked models: Gaussian spiked Wigner model, non-Gaussian spiked Wigner model and spiked Wishart model. In addition, they determined when spectral method, i.e., PCA, is optimal or sub-optimal in these models.
The single-spiked Wishart model they considered is a special case of model \eqref{eq:bipartite} with $\kappa = 1$.
% \nb{also need to put in BBP literature.}
In a slightly different line of research, Onatski et al.~\cite{onatski13,onatski14} (see also \cite{johnstone15}) derived asymptotic normality of the likelihood ratio under both single and multiple spiked Wishart models for all contiguous alternative distributions, provided that the prior on the leading eigenvectors is the uniform probability measure on the corresponding Stiefel manifold. 
The scenario they considered can be viewed as model \eqref{eq:bipartitevar} with the entries of $U$ sampled independently and identically from a standard normal distribution.
% \nb{briefly comment on the approach}
% \nb{talk about \cite{Ban16}!}
Furthermore, \citet{Alaoui17} and \citet{Alaoui18}, among other results, derived asymptotic normality results for log-likelihood ratios for single spiked Wigner and Wishart models, which allowed for priors with uniformly bounded support size on each entry of the leading eigenvector or the leading singular vector pair. 
% than the spherical uniform distribution.
Their approach is borrowed from spin glass literature and uses cavity method.
For single spiked Wishart models, their result is complementary to that in \cite{onatski13} as it required a bounded support condition on the prior and hence excluded the Gaussian prior that underpins the result in \cite{onatski13}.

% However to the best of our knowledge the priors are assumed to have bounded support which doesn't readily generalized to priors of unbounded supports.
% \nb{also need comment on the approach}

\subsection{Main contributions}
The main contributions of the present manuscript are two-folded:
\begin{enumerate}[(1)]
\item 
For both models \eqref{eq:bipartite} and \eqref{eq:bipartitevar}
and for a large collection of priors on $\Theta$ and $U$, 
we show that when the SNR is below certain threshold that depends only on $\gamma$ and second moments and sub-Gaussianity parameters of the priors, the null and the alternative distributions are asymptotically mutually contiguous and that the log-likelihood ratio has normal limits under both null and alternative as $n\to\infty$ and $p/n\to\gamma\in (0,\infty)$.
The limiting normal distributions have different means but the same variance, both of which depend only on $\gamma$ and second moments of the priors.
We allow any prior on $\Theta$ that assigns independent sub-Gaussian row vectors, and any such prior on $U$.
This allows rows of $\Theta$ (and $U$) to be i.i.d.~according to any multivariate normal distribution or any multivariate discrete/continuous distribution with bounded support, among other possibilities.
 % that results from orthonormalizing independent sub-Gaussian entries.
% In particular, this allows for independent Bernoulli or normal entries in $\Theta$ and spherically symmetric prior and orthonormalized Rademacher prior on $U$.
% We also establish comparable results for the symmetric case as $n$ tends toinfinity.
To the best of our limited knowledge, the present manuscript is the first to give such results for these multi-spiked signal detection problems beyond the case of uniform priors.

\item In either model, when the SNR is below the threshold under which we have asymptotic normality for the log-likelihood ratio, we show that under either null or alternative, the log-likelihood ratio can be decomposed as the weighted sum of a collection of statistics, defined later as bipartite signed cycles, which are asymptotically independently and normally distributed.
This provides an analysis of variance for the asymptotic log-likelihood ratio.
% In the symmetric case, the normal limit is decomposed as the sum of the normal limits of signed cycles of different lengths that are also asymptotically independent.
Such a
% an analysis of variance type
result sheds light on the source of randomness in the asymptotic likelihood ratio. 
In addition, it serves as a first step toward designing computationally efficient tests that can achieve the exact optimal power of the likelihood ratio test for the testing problems of interest.
See, for instance, the effort made in \cite{Banerjee2017} along this direction for testing \ER~random graphs against planted partition models.
% For example, for the problem of testing \ER~random graph against planted partition model, \citet{Banerjee2017} showed that when the average connection probability is much larger than $n^{-2/3}$, under both null and contiguous alternatives, signed cycles of different lengths can be approximated in probability by linear spectral statistics of a properly rescaled adjacency matrix.
% The smooth functions used in the linear spectral statistics are Chebyshev polynomials, the degrees of which match the lengths of the signed cycles.

\end{enumerate}

The approach we shall take in the present manuscript is inspired by a parallel line of research on contiguity and signal detection for \ER~random graphs and the planted partition model (i.e., symmetric two block stochastic block models).
\citet{janson95} introduced a second moment argument to study asymptotically contiguous random graph models with respect to random $d$-regular graphs where the degree parameter $d$ remains finite as the graph size tends to infinity.
In addition, he showed that the asymptotic likelihood ratios between these sparse random graph models are determined by counts of cycles on graphs.
\citet{MNS12} established a comparable set of results when studying the detection of planted partition model against \ER~random graphs in the asymptotic regime where the average degree of nodes remain finite when the graph size tends to infinity. 
They determined the exact boundary between asymptotically contiguous and singular regimes and showed that within the contiguous regime, the asymptotic log-likelihood ratio is determined by counts of cycles and has a Poisson mixture limit.
\citet{Ban16} studied the same \ER~model vs.~planted partition model detection problem in a different asymptotic regime where average degree and graph size tend to infinity together. 
Similar to \cite{MNS12}, he determined the exact boundary between contiguity and singularity while he also showed that in the contiguous regime, the asymptotic likelihood ratio is determined by a series of graph statistics called signed cycles as opposed to actual counts of cycles on graph.
The major tool in \cite{Ban16} is a Gaussian version of the second moment method in \cite{janson95}.
This second moment method also serves as the backbone of arguments in the present manuscript, while the major difficulty here lies in constructing and analyzing the collection of statistics that determine the asymptotic likelihood ratios for models \eqref{eq:bipartite} and \eqref{eq:bipartitevar}.
As we shall show later, the bipartite signed cycles will serve this purpose.

% \subsection{Other related works}
% \nb{discuss physics literature on SK model here?}

\subsection{Organization}
The rest of this manuscript is organized as the following. 
In Section \ref{sec:result}, we introduce the second moment method, define bipartite signed cycles, and state the main theorems.
Section \ref{sec:cycle} establishes the asymptotic normality of bipartite signed cycles and Section \ref{sec:proof} proves the main theorems.
The appendix presents the details of the second moment method.

% \nb{fill in later}

% \subsection{Notation}
%
% \nb{the following are ZM's suggestion on notation}
%
% \begin{itemize}
% 	\item Use $\Theta$ and $U$ for the signal components in the spiked Wishart case, $\kappa$ for the number of spikes.
% 	I suggest that for any matrix $A\in \reals^{n\times p}$, we use $A_{i*}$ to represent its $i$-th row (of size $1\times p$) and $A_{*j}$ its $j$-th column (of size $n\times 1$). We use $A'$ for matrix transposition, and use
% 	$s_1(A)\geq s_2(A)\geq \cdots$ to denote its successive singular values.
% 	Let $\mathrm{Cov}(\Theta_{i*}' \Theta_{i*}) = \Sigma_{\Theta}$ and
% 	$\mathrm{Cov}(U_{i*}' U_{i*}) = \Sigma_{U}$.
% 	\nb{my suggestion is to avoid using boldface math symbols, the only exception being $\mathbf{1}$ for indicator function}
%
% 	\item Use $B_{n,k}$ to denote bipartite signed cycles of length $2k$.
% 	Use $L_n$ for likelihood ratio.
%
% 	\item For any number $a$, $[a]$ stands for the largest integer that is no greater than $a$.
% \end{itemize}

%% file: results.tex
% !TEX root = GAU_MUL.tex

\section{Main results}
\label{sec:result}
\subsection{Preliminaries}

\paragraph{Contiguity}
% \textcolor{purple}{
For two sequences of probability measures $\bbP_n$ and $\bbQ_n$ defined on $\sigma$-fields $(\Omega_n,\calF_n)$,
we say that $\bbQ_n$ is contiguous with respect to $\bbP_n$, denoted by $\bbQ_n \triangleleft \bbP_n$, if for any event sequence $A_n$, $\bbP_n(A_n)\to 0$ implies $\bbQ_n(A_n)\to 0$.
We say that they are (asymptotically) mutually contiguous, denoted by $\bbP_n \triangleleft\triangleright \bbQ_n$, if both $\bbQ_n\triangleleft \bbP_n$ and $\bbP_n\triangleleft \bbQ_n$ hold.
We refer interested readers to \citet{LeCam} and \citet{LeCam00} for general discussions on contiguity.
% }

% \textcolor{purple}{
To establish our main results, we rely on the following proposition for establishing contiguity and asymptotic normality of log-likelihood ratios.
For any two probability measure $\bbP$ and $\bbQ$ on the same probability space, we write $\bbQ\ll \bbP$ when $\bbQ$ is absolutely continuous with respect to $\bbP$.
% }

\begin{prop}[Janson's second moment method] %[\cite{Ban16}]
\label{prop:norcont}
Let $\bbP_n$ and $\bbQ_n$ be two sequences of probability measures such that for each $n$, both are defined on the common $\sigma$-algebra $(\Omega_n, \calF_n)$.
Suppose that for each $i\geq 1$, $W_{n,i}$ are random variables defined on $(\Omega_n,\calF_n)$.
Then the probability measures $\bbP_n$ and $\bbQ_n$ are asymptotically mutually contiguous if the following conditions hold simultaneously:
\begin{enumerate}[(i)]
\item $\bbQ_n$ is absolutely continuous with respect to $\bbP_n$
% $\bbQ_n\ll \bbP_n$
for each $n$;
%\item $\bbP_{n} \ll \bbQ_{n}$ and $\bbQ_n\ll \bbP_n$ for each $n$;
\item For any fixed $k\ge 1$, one has $\left( W_{n,1},\ldots, W_{n,k} \right)|\mathbb{P}_{n} \stackrel{d}{\to} \left(Z_{1}, \ldots, Z_{k}\right) $  and $\left( W_{n,1},\ldots, W_{n,k} \right)|\mathbb{Q}_{n} \stackrel{d}{\to} \left(Z'_{1}, \ldots, Z'_{k}\right)$.

\item $Z_{i}\sim N(0,\sigma_{i}^{2})$ and $Z'_{i} \sim N(\mu_{i},\sigma_{i}^{2})$ are sequences of independent random variables.

\item The likelihood ratio statistic $Y_n = \frac{\mathrm{d}\bbQ_n}{\mathrm{d}\bbP_n}$ satisfies 
\begin{equation}
	\label{eq:lr-square}
\limsup_{n\to\infty}\Expect_{\bbP_n}\qth{Y_n^2} \leq
\exp\sth{\sum_{i=1}^\infty \frac{\mu_i^2}{\sigma_i^2}} < \infty.
\end{equation}
\end{enumerate}
Furthermore, we have that
% \begin{enumerate}[(a)]
% \item
under $\bbP_n$,
\begin{equation}
	\label{eq:lr-limit}
Y_n \stackrel{d}{\to} \exp\sth{\sum_{i=1}^\infty \frac{\mu_i Z_i - \frac{1}{2}\mu_i^2}{\sigma_i^2}}.
\end{equation}
% \item
and that given any $\epsilon,\delta>0$ there exists a natural number $K=K(\delta,\epsilon)$ such that for any sequence $n_l$ there is a further sub-sequence $n_{l_m}$ such that 
\begin{equation}
	\label{eq:janson-decomp}
\limsup_{m\to\infty} \mathbb{P}_{n_{l_m}}\left[ \left| \log(Y_{n_{l_m}}) - \left\{\sum_{k=1}^{K} \frac{2\mu_{k}\left(W_{n_{l_m},k}\right)-\mu_{k}^{2}}{\sigma_{k}^{2}} \right\}\right|\ge \epsilon \right] \le  \delta.
\end{equation}
%for a suitable sequence $m_{n}$ growing to $\infty$.
% \end{enumerate}
\end{prop}
%\begin{proof}
%\nb{give a pointer to its proof in an appendix.}
%\end{proof}

\begin{remark}
The proposition can be viewed as a Gaussian version of Theorem 1 in Janson \cite{janson95} which dealt with convergence to a Poisson mixture. 
In addition, it generalizes Proposition 3.4 in \cite{Ban16} where a more specific version of it appeared with an additional redundant condition ($\bbP_n$ being absolutely continuous with respect to $\bbQ_n$) and without conclusion \eqref{eq:janson-decomp}. 
% However to avoid confusions, we give a complete proof here.
\end{remark}

\paragraph{Bipartite signed cycles}
In view of Proposition \ref{prop:norcont},
our proofs rely on finding out a class of random variables which are ``asymptotically sufficient" for determining the likelihood ratio. 
To this end, we define the following set of statistics.
% We call these statistics bipartite signed cycles in the bipartite case and signed cycles in the symmetric case. The  signed cycles were first studied in \citet{Ban16} for the case of planted partition model. At first we define the bipartite signed cycles
\begin{definition}[Bipartite signed cycle of length $2k$]
For each $k\in [n\wedge p]$, we define the bipartite signed cycle of length $2k$ as 
\begin{equation}\label{DEF_BICYCLE}
\BSC_{n,k}= 
\frac{1}{n^{k}} \sum_{i_0, j_0, i_1, j_1, \ldots,i_{k-1}, j_{k-1}} X_{i_0,j_0}X_{i_1,j_0} X_{i_1,j_1} X_{i_2,j_1} \ldots X_{i_{k-1},j_{k-1}} X_{i_0,j_{k-1}}
\end{equation}
 where $i_0,i_1,\ldots,i_{k-1}\in [n]$ are all distinct, and so are $j_0,j_1,\dots,j_{k-1}\in [p]$.
\end{definition} 

As we shall show later, for both models \eqref{eq:bipartite} and \eqref{eq:bipartitevar}, the collection of bipartite signed cycles of increasing lengths determines the asymptotic likelihood ratio, 
at least for a large collection of prior distributions on $\Theta$ and $U$ which we now define.

\paragraph{Sub-Gaussian prior distributions}
We first define sub-Gaussian random vectors and their variance proxies.

\begin{definition}
	% [Variance proxy]
	\label{def:varproxy}
Suppose $X$ is a random vector of dimension $d$. We say the random vector $X$ is sub-Gaussian with variance proxy $\wt{\Sigma}_{X}$ if $\Expect[X]=0$ and $\Expect[\exp\left( t'X \right)]\leq \exp\left( \frac{1}{2} t' \wt{\Sigma}_{X} t \right)$ for any $t \in \mathbb{R}^{d}$. 
Here $\wt{\Sigma}_{X}$ is a non-negative definite matrix.
% \nb{change it to vector-matrix version?}
\end{definition}

By definition, if $\wt\Sigma_X$ is a variance proxy for $X$, then so is any matrix $\wt\Sigma$ such that $\wt\Sigma - \wt\Sigma_X$ is non-negative definite.
For any multivariate normal distribution the variance proxy matches with the true covariance matrix.
If $X$ is a random vector with i.i.d.~Rademacher entries then $X$ is sub-Gaussian with variance proxy $I_{d}$. 
Furthermore, if $X$ is sub-Gaussian with variance proxy $\wt{\Sigma}_{X}$ then $AX$ is sub-Gaussian with variance proxy $A'\wt\Sigma_X A$ for any $A$. 
Finally, if $X$ is a random variable taking values within $[a,b]$, then $X-\Expect[X]$ is sub-Gaussian with variance proxy $\frac{1}{4}\left({a-b}\right)^2$.

\begin{definition}[Sub-Gaussian prior]
\label{def:prior}
For any given number $\kappa < \min(n,p)$, let $\calP(n,\kappa, \Sigma_\Theta, \wt\Sigma_\Theta)$ be the collection of all priors $\pi_\Theta$ on $\Theta$ such that under $\pi_\Theta$, the row vectors $\{\Theta_{i*}:i\in [n] \}$ are i.i.d.~sub-Gaussian random vectors in $\reals^\kappa$ with mean zero, covariance matrix $\Sigma_\Theta$ and variance proxy $\wt\Sigma_\Theta$.
Let $\calP(p,\kappa, \Sigma_U, \wt\Sigma_U)$ be defined analogously for $U$.
\end{definition}

% Our proofs assume that the prior on $\Theta$ and $U$ introduced in the introduction is sub-Gaussian with certain variance proxy. Formally we assume
% \begin{assumption}\label{ASS:SUB}
% \begin{enumerate}[(i)]
% \item $\Theta_{1,*} \in \mathbb{R}^{\kappa}$ is a sub-Gaussian random vector with variance proxy $\wt{\Sigma}_{\Theta}$ according to Definition \ref{def:varproxy}. Note that as mentioned in the introduction we consider the the variance covariance matrix of $\Theta_{1,*}$ to be $\Sigma_{\Theta}$. We call this class of priors $\pi_{\Theta}$ for $\Theta$, $\calP(n,\kappa,\wt{\Sigma}_{\Theta})$.
% \item $U_{1,*} \in \mathbb{R}^{\kappa}$ is a sub-Gaussian random vector with variance proxy $\wt{\Sigma}_{U}$  according to Definition \ref{def:varproxy}. As in the case of $\Theta$, here we denote $\Sigma_{U}$ to be the variance covariance matrix of $U_{1,*}$. We call this class of priors $\pi_{U}$ for $U$, $\calP(n,\kappa,\wt{\Sigma}_{U})$.
% \end{enumerate}
% \end{assumption}

\subsection{Statement of theorems}

We first state the theorem for testing problem \eqref{eq:bipartite}.
For any matrix $A$, let $\|A\|_2$ be its spectral norm.
For any event $E$, let $\mathbf{1}_E$ be its indicator function.

\begin{theorem}\label{thm:main}
Consider the testing problem defined in (\ref{eq:bipartite}) with the $\Theta$ prior $\pi_\Theta\in \calP(n,\kappa,\Sigma_\Theta,\wt{\Sigma}_{\Theta})$ 
and the $U$ prior 
$\pi_U\in \calP(p,\kappa,\Sigma_U, \wt{\Sigma}_{U})$.
Denote the null distribution by $\bbP_{0,n}$, the alternative distribution $\bbP_{1,n}$ and the likelihood ratio $L_{n}=  \frac{\diff\mathbb{P}_{1,n}}{\diff\mathbb{P}_{0,n}}$. 
Suppose as $n\to\infty$, $p/n\to\gamma\in (0,\infty)$ while $\kappa,\Sigma_{\Theta},\Sigma_{U},\wt{\Sigma}_{\Theta},\wt\Sigma_{U}$ remain fixed. The following hold whenever $\|\wt{\Sigma}_{\Theta}\wt{\Sigma}_{U} \|_{2} \|\Sigma_{\Theta}\Sigma_{U} \|_{2} < \gamma$.
\begin{enumerate}
\item $\mathbb{P}_{0,n}$ and $\mathbb{P}_{1,n}$ are asymptotically mutually contiguous.
\item Under $H_0$,
\begin{equation}
\label{eq:liklim}
L_{n} \stackrel{d}{\to} \exp\left\{ \sum_{k=1}^{\infty} \frac{2\mu_{k}Z_{k}-\mu_{k}^{2}}{4k\gamma^{k}} \right\}
\end{equation}
where $Z_{k}$ are independent $N\left(0, 2k \gamma^{k}\right)$ random variables and for any $k$, $\mu_{k}= \Tr\left((\Sigma_{\Theta}\Sigma_{U})^k\right)$.
In other words, under $H_0$, $\log(L_n)\stackrel{d}{\to} N(-\frac{1}{2}\sigma_b^2, \sigma_b^2)$ with 
\begin{equation}
	\label{eq:sigma0}
	\sigma_b^2 = \sum_{k=1}^{\infty} \frac{\mu_{k}^{2}}{2k\gamma^{k}} = \frac{1}{2}\sum_{i=1}^\kappa \sum_{j=1}^\kappa \log\left(1 - \frac{h_ih_j}{\gamma}\right)
\end{equation}
where $h_1\geq\cdots\geq h_\kappa$ are the eigenvalues of $\Sigma_\Theta\Sigma_U$.
Under $H_1$, we have \eqref{eq:liklim} with $Z_k\stackrel{ind}{\sim} N(\mu_{k}, 2k\gamma^k)$ and $\log(L_n)\stackrel{d}{\to} N(\frac{1}{2}\sigma_b^2, \sigma_b^2)$.
\item Further under both null and alternative the log-likelihood ratio satisfies the following ANOVA type decomposition:
\begin{equation}
	\label{eq:anova}
\log(L_{n})- \sum_{k=1}^{m_{n}} \frac{2\mu_{k}\left(\BSC_{n,k}-p\Indc_{k=1}\right)-\mu_{k}^2}{4k\gamma^{k}} \stackrel{p}{\to} 0
\end{equation}
where $m_{n}$ is any sequence growing to $\infty$ at a rate $o(\sqrt{\log{n}})$ .
\end{enumerate}
\end{theorem}

\iffalse
\begin{remark}
 Theorem \ref{thm:main} in this paper implies Theorem 1 in \citet{Alaoui18}. Our notations are slightly different than \citet{Alaoui18}. In that paper one consider the following testing problem. Suppose the data matrix is of order $N \times M$ and we consider the following hypothesis testing problem. 
 \[
 H_{0}: X= E, ~~\text{vs.} ~~ H_{1}: X= \sqrt{\frac{\beta}{N}}uv' + E.
 \] 
 where $u$ and $v$ are i.i.d. mean $0$ variance $1$ and $\frac{M}{N} \to \alpha$. Under the assumption that $u$ and $v$ having bounded support with radius $K_{u}$ and $K_{v}$, they prove that the log-likelihood under null and alternative ratio converges to $N\left(\pm\frac{1}{4}\log\left( 1- \alpha\beta^{2} \right),
\frac{1}{2}\log\left(1- \alpha\beta^2\right)\right)$ whenever $\alpha\beta^2K_{u}^{4}K_{v}^{4}<1$. Using the fact that any random variable having support with in $[a,b]$ is sub Gaussian with variance proxy $\frac{(a-b)^2}{4}$ and taking $\Theta=\sqrt{\beta}u$ and $U=v$ we have the log-likelihood ratio under null and alternative converges to $N\left(\pm\frac{1}{4}\log\left( 1- \alpha\beta^{2} \right),
\frac{1}{2}\log\left(1- \alpha\beta^2\right)\right)$ whenever $\alpha\beta^2K_{u}^2K_{v}^2<1$. As $K_{u}$ and $K_{v}$ are greater or equal to $1$, we have Theorem 1 in this paper implying Theorem 1 in \citet{Alaui18}.
\end{remark}
\fi

Next we state the counterpart of Theorem \ref{thm:main} for the testing problem in (\ref{eq:bipartitevar}).
\begin{theorem}
	\label{thm:mainvar}
Consider the testing problem defined in (\ref{eq:bipartitevar}).
Denote the null distribution by $\bbP_{0,n}$, the alternative distribution $\bbP_{1,n}$ and the likelihood ratio $L_{n}=  \frac{\diff\mathbb{P}_{1,n}}{\diff\mathbb{P}_{0,n}}$. 
Under the condition of Theorem \ref{thm:main}, whenever 
\begin{equation*}
\|\Sigma_{U}^{-1/2}\wt{\Sigma}_{U}\Sigma_{U}^{-1/2}\wt{\Sigma}_{\Theta} \|_{2}
\| \Sigma_{\Theta} \|_{2} < \gamma,
\end{equation*}
the three conclusions of Theorem \ref{thm:main} hold with $\mu_k = \Tr(\Sigma_\Theta^k)$ and $h_1\geq\cdots\geq h_\kappa$ the eigenvalues of $\Sigma_\Theta$.
\end{theorem}

It is not surprising that the results for 
% worth mentioning that
the testing problems (\ref{eq:bipartite}) and (\ref{eq:bipartitevar}) are essentially the same. 
By law of large number, one has $\|\sqrt{p}(UU')^{-{1}/{2}}- \Sigma_{U}^{-{1}/{2}}\|_{2} \stackrel{p}{\to} 0$. 
Hence $(UU')^{-{1}/{2}}$ is essentially same as $\frac{1}{\sqrt{p}} \Sigma_{U}^{-{1}/{2}}$ for large values of $p$. 
In addition, the distribution of $X$ in \eqref{eq:bipartitevar} remains unchanged if we replace $U$ with $U\Sigma_U^{-1/2}$.
Hence the testing problem is essentially the same as the unnormalized version (\ref{eq:bipartite}) with $\Sigma_{U}= I_{\kappa}$.

\paragraph{Sphericity test against multi-spiked Wishart covariance matrix}

Suppose $\pi^0_\Theta$ assigns i.i.d.~$N_\kappa(0, I_\kappa)$ rows vectors in $\Theta$ and $\pi^0_U$ does the same on $U$.
Then $V = U(U'U)^{-1/2}$ follows the uniform distribution on the Stiefel manifold $O(p,\kappa)$, and \eqref{eq:bipartitevar} reduces to the high-dimensional sphericity testing problem considered in \cite{onatski14}, 
since in this case, the full data likelihood ratio reduces to the likelihood ratio of the eigenvalues of the sample covariance matrix.
Since $\pi^0_\Theta\in \calP(n,\kappa,I_\kappa, I_\kappa)$ and $\pi^0_U\in \calP(p,\kappa,I_\kappa, I_\kappa)$, we obtain the following corollary of Theorem \ref{thm:mainvar} which reconstructs the key result in \cite{onatski14} for normal data.
However, the proof approach we take is completely different from that used in \cite{onatski14}.
 
\begin{cor}\label{cor:omh}
Let $X_1,\dots, X_n \stackrel{iid}{\sim} N_p(0, \Sigma)$.
Consider testing $H_0: \Sigma = I_p$ vs.~$H_1: \Sigma = I_p + V'HV$ where $H = \mathrm{diag}(h_1,\dots, h_\kappa)$ with $h_1\geq \cdots \geq h_\kappa > 0$ and $V$ follows the uniform distribution on $O(p,\kappa)$. 
Denote the null distribution by $\bbP_{0,n}$, the alternative distribution $\bbP_{1,n}$ and the likelihood ratio $L_{n}=  \frac{\diff\mathbb{P}_{1,n}}{\diff\mathbb{P}_{0,n}}$. 
Suppose as $n\to\infty$, $p/n\to\gamma\in (0,\infty)$ while $\kappa$ remains fixed. 
The following hold whenever $h_1 < \sqrt{\gamma}$.
\begin{enumerate}
\item $\mathbb{P}_{0,n}$ and $\mathbb{P}_{1,n}$ are asymptotically mutually contiguous.
\item Under $H_0$, \eqref{eq:liklim} holds
% \begin{equation}
% \label{eq:liklim}
% L_{n} \stackrel{d}{\to} \exp\left\{ \sum_{k=1}^{\infty} \frac{2\mu_{k}Z_{k}-\mu_{k}^{2}}{4k\gamma^{k}} \right\}
% \end{equation}
where $Z_{k}$ are independent $N\left(0, 2k \gamma^{k}\right)$ random variables and for any $k$, $\mu_{k}= \Tr\left(H^k\right)$.
In other words, under $H_0$, $\log(L_n)\stackrel{d}{\to} N(-\frac{1}{2}\sigma_b^2, \sigma_b^2)$ with 
\begin{equation*}
	% \label{eq:sigma0}
	\sigma_b^2 
	% = \sum_{k=1}^{\infty} \frac{\mu_{k}^{2}}{2k\gamma^{k}}
	= \frac{1}{2}\sum_{i=1}^\kappa \sum_{j=1}^\kappa \log\left(1 - \frac{h_ih_j}{\gamma}\right).
\end{equation*}
% where $h_1\geq\cdots\geq h_\kappa$ are the eigenvalues of $\Sigma_\Theta\Sigma_U$.
Under $H_1$, we have \eqref{eq:liklim} with $Z_k\stackrel{ind}{\sim} N(\mu_{k}, 2k\gamma^k)$ and $\log(L_n)\stackrel{d}{\to} N(\frac{1}{2}\sigma_b^2, \sigma_b^2)$.
\item Further under both null and alternative the log-likelihood ratio satisfies \eqref{eq:anova}
where $m_{n}$ is any sequence growing to $\infty$ at a rate $o(\sqrt{\log{n}})$ .
\end{enumerate}
% Suppose the co-ordinates $\{u_{jl}\}_{1\le l \le p}$  are taken to be i.i.d. mean $0$ variance $1$ and sub-Gaussian with variance proxy $\sigma^{2}$. We consider the case when $\frac{p}{n} \to \gamma \in (0,\infty)$ and $\lambda_{1}<\frac{\sqrt{\gamma}}{\sigma}$
%  Let $L_{n}=  \frac{\diff\mathbb{P}_{1,n}}{\diff\mathbb{P}_{0,n}} $. Then
%  \begin{enumerate}
%  \item $\mathbb{P}_{0}$ and $\mathbb{P}_{1}$ are asymptotically contiguous if  $\lambda_1 < \frac{\sqrt{\gamma}}{\sigma}$.
%  \item Whenever $\lambda_1 < \frac{\sqrt{\gamma}}{\sigma}$, under $H_0$,
%  \[
%  L_{n} \stackrel{d}{\to} \exp\left\{ \sum_{k=1}^{\infty} \frac{2\left(\sum_{l=1}^{\kappa} \lambda_{l}^{k}\right)Z_{k}-(\sum_{l=1}^{\kappa} \lambda_{l}^{k})^{2}}{4k\gamma^{k}} \right\}
%  \]
%  where $Z_{k}$ are independent $N\left(0, 2k \gamma^{k}\right)$ random variables.
%  \item Under both null and alternative, one also has for any sequence $m_{n}= o(\sqrt{n})$ with $n$
%  \[
% \log( L_{n})- \sum_{k=1}^{m_{n}} \frac{2\left(\sum_{l=1}^{\kappa} \lambda_{l}^{k}\right)\left(\BSC_{n,k}- p\Indc_{k=1}\right)-(\sum_{l=1}^{\kappa} \lambda_{l}^{k})^{2}}{4k\gamma^{k}} \stackrel{p}{\to} 0.
%  \]
%  \end{enumerate}
 \end{cor}
 
\begin{remark}
Since our main theorems depend on the priors only through covariance matrices and sub-Gaussian variance proxies, Corollary \ref{cor:omh} holds for any prior $\pi_U\in \calP(p,\kappa, I_\kappa, I_\kappa)$, e.g., the prior that assigns entries of $U$ with i.i.d.~Rademacher random variables.
\end{remark}

%% file: cycle.tex
\section{Asymptotic normality of bipartite signed cycles}
\label{sec:cycle}

We now give the limiting distribution of the statistics $B_{n,k}$'s defined in (\ref{DEF_BICYCLE}). 
These statistics serve the purpose of the $W_{n,i}$'s in Proposition \ref{prop:norcont} when it comes to proving Theorems \ref{thm:main} and \ref{thm:mainvar}.
 \begin{prop}[Bipartite signed cycles]\label{prop:bicyclelim}
Consider both testing problems \eqref{eq:bipartite} and \eqref{eq:bipartitevar}.
The following results hold:
 \begin{enumerate}[(i)]
 \item Under $H_0$, when $1\le  k_1 < \ldots< k_l=o(\sqrt{\log n})$, 
 \begin{equation}
 \left(\frac{B_{n,k_1}- p\Indc_{k_{1}=1}}{\sqrt{2k_1\gamma^{k_1}}},\ldots, \frac{B_{n,k_l}}{\sqrt{2k_l\gamma^{k_l}}}\right) \stackrel{d}{\to} N_l(0,I_l).
 \end{equation}
 \item Under $H_1$, 
 % conditional on $\Theta$ and $U$ \nb{NO, not true! it is marginal for all $k_l\geq 2$},
when $1\le k_1 < \ldots< k_l=o(\sqrt{\log n})$,
 \begin{equation} \left(\frac{B_{n,k_1}-p\Indc_{k_{1}=1}-\mu_{k_1}}{\sqrt{2k_1\gamma^{k_1}}},\ldots, \frac{B_{n,k_l}-\mu_{k_l}}{\sqrt{2k_l\gamma^{k_l}}}\right)
% |_{\mathbb{P}_{1},\Theta_{i,1},\Theta_{i,2},\mathbf{u}_{1},\bla~|~ 1\le i \le n}
 \stackrel{d}{\to} N_l(0,I_l),
 \end{equation}
 where for testing problem \eqref{eq:bipartite},
\begin{equation}
	\label{eq:mu-k-b}
\begin{aligned}
\mu_{k} & = \sum_{l_{1}\ldots,l_{2k}\in \{1,\ldots,\kappa\}^{2k}} \Sigma_{\Theta}(l_{1},l_{2k})\Sigma_{U}(l_{1},l_{2})\Sigma_{\Theta}(l_{2},l_{3})\ldots \Sigma_{U}(l_{2k-1},l_{2k})\\
& =\Tr\big((\Sigma_{\Theta}\Sigma_{U})^k\big),	
\end{aligned}
\end{equation}
and for testing problem \eqref{eq:bipartitevar}
\begin{equation}
\mu_k = \Tr(\Sigma_\Theta^k).
\end{equation}
 %(This is true if you had $\mathbf{u}_{2}$ instead of $\bla$ but needs a little checking in the $\bla$ case. I shall do it soon.)
 \end{enumerate}
 \end{prop}

\iffalse
 \begin{prop}[Signed cycles]\label{prop:signedcyclelim}
  \begin{enumerate}[(i)]
 \item Under $H_0$, when $1\le  k_1 < \ldots< k_l=o(\sqrt{\log(n)})$, 
 \begin{equation}
 \left(\frac{{C}_{n,k_1}- n\Indc_{k_{1}=2}}{\sqrt{2k_1}},\ldots, \frac{{C}_{n,k_l}-n\Indc_{k_{l}=2}}{\sqrt{2k_l}}\right) \stackrel{d}{\to} N_l(0,I_l).
 \end{equation}
 \item Under $H_1$, 
 % conditional on $\Theta$ and $U$ \nb{no, not ture!},
 when $1\le k_1 < \ldots< k_l=o(\sqrt{\log(n)})$,
 \begin{equation}
 \left(\frac{{C}_{n,k_1}-n\Indc_{k_{1}=2}-\mu_{k_1}}{\sqrt{2k_1}},\ldots, \frac{{C}_{n,k_l}-n\Indc_{k_{l}=2}-\mu_{k_l}}{\sqrt{2k_l}}\right)
% |_{\mathbb{P}_{1},\Theta_{i,1},\Theta_{i,2},\mathbf{u}_{1},\bla~|~ 1\le i \le n}
 \stackrel{d}{\to} N_l(0,I_l),
 \end{equation}
 where 
 \begin{equation}\label{eq:mu-k-s}
 \mu_{k} \stackrel{p}{\to} \sum_{l_{1},\ldots,l_{k} \in \{ 1,2,\ldots,\kappa\}^{k}} \Sigma_{U}(l_{1},l_{k})\Sigma_{U}(l_{1},l_{2})\ldots \Sigma_{U}(l_{k-1},l_{k})=\Tr(\Sigma_{U}^{k}).
 \end{equation}
 \end{enumerate}
 \end{prop}
\fi

\subsection{Preliminaries}
 
% Proof of Proposition \ref{prop:signedcyclelim} is similar to the proof of Proposition 4.1 in Banerjee \cite{Ban16}. So we shall omit this proof.
The proof of the foregoing proposition is inspired by the remarkable paper \cite{AZ05}. 
% However, the model in this case is simpler which makes the proof less cumbersome.
The fundamental idea is to prove the asymptotic normality by using the method of moments and showing that moments of the limiting distributions satisfy Wick's formula. 
We first state the method of moments.
\begin{lemma}\label{lem:mom}
Let $Y_{n,1},\ldots, Y_{n,l}$ be a random vector of $l$ dimension. Then $(Y_{n,1},\ldots, Y_{n,l}) \stackrel{d}{\to} (Z_1,\ldots,Z_{l})$ if the following conditions are satisfied:
\begin{enumerate}[(i)]
\item 
% \begin{equation}\label{eqn:momcond}
$\lim_{n \to \infty}\Expect[X_{n,1}\ldots X_{n,m}]$
% \end{equation}
exists for any fixed $m$ and $X_{n,i} \in \{ Y_{n,1},\ldots,Y_{n,l} \}$ for $1\le i \le m$.
\item(Carleman's Condition)\cite{Carl26}
\[
\sum_{h=1}^{\infty} \left(\lim_{n \to \infty}\Expect[X_{n,i}^{2h}]\right)^{-\frac{1}{2h}} =\infty ~~ \forall ~ 1\le i \le l.
\]
\end{enumerate} 
Further,  
\[
\lim_{n \to \infty}\Expect[X_{n,1}\ldots X_{n,m}]= \Expect[X_{1}\ldots X_{m}].
\]
Here $X_{n,i} \in \{ Y_{n,1},\ldots,Y_{n,l} \}$ for $1\le i \le m$ and $X_{i}$ is the in distribution limit of $X_{n,i}$. 
In particular, if $X_{n,i}= X_{n,j}$ for some $i\neq j \in \{1,\ldots,l \}$ then $X_{i}= X_{j}$. 
\end{lemma} 
% The method of moments is very well known and much useful in probability theory. We omit its proof.

% \noindent
Now we state Wick's formula for Gaussian random variables which was first proved by \citet{I18} and later on introduced by \citet{W50} in the physics literature.
 
\begin{lemma}[Wick's formula \cite{W50}]
	\label{lem:wick}
Let $(Y_1,\ldots, Y_{l})$ be a multivariate mean $0$ random vector of dimension $l$ with covariance matrix $\Sigma$ (possibly singular). Then $(Y_1,\ldots, Y_{l})$ is jointly Gaussian if and only if for any integer $m$ and $X_{i} \in \{ Y_1,\ldots,Y_{l} \}$ for $1\le i \le m$
\begin{equation}\label{eqn:wick}
\Expect[X_1\ldots X_{m}]=\left\{
\begin{array}{ll}
  \sum_{\eta} \prod_{i=1}^{\frac{m}{2}} \Expect[X_{\eta(i,1)}X_{\eta(i,2)}] & ~ \text{for $m$ even}\\
  0 & ~ \text{for $m$ odd.}
\end{array}
\right.
\end{equation}
Here $\eta$ is a partition of $\{1,\ldots,m \}$ into $\frac{m}{2}$ blocks such that each block contains exactly $2$ elements and $\eta(i,j)$ denotes the $j$th element of the $i$th block of $\eta$ for $j=1,2$.
\end{lemma}

The proof of the aforesaid lemma is omitted. 
However, we note that the random variables $Y_1,\ldots, Y_{l}$ may be the same. In particular, taking $Y_1=\cdots = Y_{l}$, Lemma \ref{lem:wick} provides a description of the moments of multivariate Gaussian random variables.

\subsection{Proof of Proposition \ref{prop:bicyclelim}}
\label{subsec:proof-bicycle}
In this part, we focus on the proof of Proposition \ref{prop:bicyclelim} for testing problem \eqref{eq:bipartite}.
In view of the discussion after the statement of Theorem \ref{thm:mainvar}, the modification of the formula for $\mu_k$ for \eqref{eq:bipartitevar} is natural as we could in some sense treat this case similarly to \eqref{eq:bipartite} with $\Sigma_U = I_\kappa$.
 % For the testing problem (\ref{eq:bipartitevar}) the proof of $\mu_{k} \stackrel{p}{\to} \Tr(\Sigma_{\Theta}^k) $ is analogous.
In particular, we could modify the proof below by considering a sequence of high probability events $\Omega_{n}$ such that $\maxnorm{\Indc_{\Omega_n}\sqrt{p}(UU')^{-{1}/{2}}-\Sigma_U^{-{1}/{2}}} \le \delta_{n} \to 0$ and then establish all the weak limits on $\Omega_{n}$.
 % to $\wt{U}= \Sigma^{-\frac{1}{2}}U$ instead.
Here and after, for any matrix $A$, $\|A\|_{\max} = \max_{i,j}|A_{ij}|$ is its vector $\ell_\infty$ norm.

% The proof of Proposition \ref{prop:signedcyclelim} is analogous to that of Proposition 4.1 in \cite{Ban16} in structure and to that of Proposition \ref{prop:bicyclelim} below when it comes to dealing with the within-row dependence of the spins.
% It is hence omitted.

% With Lemma \ref{lem:mom} and \ref{lem:wick} in hand, we now jump into the proof of Proposition \ref{prop:bicyclelim}.\\

% \textbf{Proof of Proposition \ref{prop:bicyclelim}}\\

\paragraph{Additional notation and definition}
% At first we introduce some notations and some terminologies.
Given a set $\mathcal{S}$, an $\mathcal{S}$ letter $s$ is simply an element $s\in \mathcal{S}$. 
With two sets $\mathcal{S}_{1}$ and $\mathcal{S}_{2}$, a \emph{bi-word} for $\mathcal{S}_{1}$ and $\mathcal{S}_{2}$ is defined as an alternating ordered sequence of letters 
% from $\mathcal{S}_{1}$ and $\mathcal{S}_{2}$
where the letters at odd positions come from $\mathcal{S}_{1}$ and the letters at even positions come from $\mathcal{S}_{2}$;
The final letter is required to come from $\mathcal{S}_{1}$. 
We call the letters from $\mathcal{S}_{1}$ \emph{type I} and those from $\mathcal{S}_{2}$ \emph{type II}. 
Given any bi-word $w$, the $i$th type I letter is denoted by $\alpha_{i}$ and the $i$th type II letter by $\beta_{i}$.
As a convention, we start the subscripts for letters in a bi-word with $0$. 
Observe that any bi-word $w$ starts from and ends with a type I letter and so the total number of letters in $w$ is always odd. 
In particular, any bi-word $w$ looks like $(\alpha_{0},\beta_{0},\alpha_{1},\beta_{1},\ldots, \alpha_{k})$.
% For any bi-word $w= (\alpha_{0},\beta_{0},\alpha_{1},\beta_{1},\ldots, \alpha_{k})$
We use $l(w)=2k+1$ 
% \nb{double check how this unusual definition affects later calculation!}
% \nb{this also contradicts the latter use of $l$ to denote length of cycle and of word in general!!}
to denote the length of $w$.
% the number of appearances of type I letters in $w$.
A bi-word is called \emph{closed} if $\alpha_{0}=\alpha_{k}$. 

Throughout the proof, we take $\mathcal{S}_{1}=\{1,\ldots,p \}$ and $\mathcal{S}_{2}= \{ 1,\ldots, n \}$.  
The bipartite graph induced by a bi-word $w= (\alpha_{0},\beta_{0},\alpha_{1},\beta_{1},\ldots, \alpha_{k})$ is denoted by $G_w$.  
It is defined as follows. 
One treats the letters $(\alpha_{0},\beta_{0},\alpha_{1},\beta_{1},\ldots, \alpha_{k})$ as nodes and one puts an edge between $\alpha_{i}$ and $\beta_{j}$ whenever $|i-j|=1$. 
Observe that for any closed bi-word $w$, $G_{w}$ is a cycle of even length\footnote{Cycles of odd length in a bipartite graph do not exist.}. 
Two bi-words $w_{1}$ and $w_{2}$ are called \emph{paired} if the graphs $G_{w_{1}}$ and $G_{w_{2}}$ are the same. 
For a closed bi-word $w=(\alpha_{0},\beta_{0},\alpha_{1},\beta_{1},\ldots, \alpha_{k})$, its \emph{mirror image} is $\tilde{w}= (\alpha_{k},\beta_{k-1},\alpha_{k-1},\beta_{k-2},\ldots, \alpha_{0})$. 
Furthermore, for a cyclic permutation $\tau$ of the set $\{0,1,\ldots,k-1 \}$ and a closed bi-word $w$, we define $w^{\tau}:=(\alpha_{\tau(0)},\beta_{\tau(0)},\alpha_{\tau(1)},\beta_{\tau(1)},\ldots,\beta_{\tau(k-1)},\alpha_{\tau(0)})$. 
If two closed bi-words $w_{1}$ and $w_{2}$ are paired, then there exists a cyclic permutation $\tau$ such that either $w_{1}^{\tau}=w_{2}$ or $\tilde{w}_{1}^{\tau}=w_{2}$. 

\begin{remark}
These bi-words are not fundamentally different from the words defined in \cite{AZ05} and \cite{AGZ}. 
In particular, they form a restricted class of words where the alphabet set is taken to be $\mathcal{S}_{1} \cup \mathcal{S}_{2}$. 
Hence all the properties of the words can be derived with minimal modifications of the proofs in \cite{AZ05} and \cite{AGZ}. 
% We shall not prove any of these results here.
\end{remark}

We call an ordered tuple of $m$ words $(w_1,\ldots,w_m)$ a \emph{sentence}. For any sentence $a= (w_1,\ldots w_m)$, $G_{a}=(V_a,E_a)$ is the graph with $V_a=\cup_{i=1}^{m}V_{w_i}$ and $E_{a}= \cup_{i=1}^{m} E_{w_i}$.
A sentence $a$ is called a \emph{weak CLT sentence} if each edge in $G_a$ is traversed at least twice.
By Lemma 4.10 of \cite{AZ05},
The following lemma gives a bound on the number of weak CLT sentences.
For any numbers $b$ and $c$, $b\vee c = \max(b,c)$ and $b\wedge c = \min(b,c)$.
 % having strictly less than $\lfloor \frac{m}{2} \rfloor$ connected components.
\begin{lemma}\label{lem:appendix}
Let $\mathcal{A}_{t} = \mathcal{A}_t(l_1,\dots, l_m)$ be the set of weak CLT sentences such that for each $a \in \mathcal{A}_{t}$, it consists of $m$ words of lengths $l_1,\dots, l_m$ respectively and $\#V_{a}=t$. Then 
\begin{equation}\label{eqn:bddweakclt}
\# \mathcal{A}_{t} \le 2^{\sum_{i} l_i}\left(C_1\sum_{i}l_i\right)^{C_2m}\left(\sum_{i}l_i\right)^{3(\sum_{i}l_i-2t)}n^{t}\left(\gamma \vee 1 \right)^{t}.
\end{equation}
\end{lemma}
\begin{proof}
The proof of this lemma is almost identical to the proof of Lemma 4.3 in \cite{Ban16}. The only difference is in the possible choices of vertices of $V_{a}$. Here this choice will be $n^{t_{1}}p^{t_{2}}$ where $t_{1}+ t_{2}= t$ and $t_{1}$ is the number of vertices which are from $\mathcal{S}_{1}$ and $t_{2}$ is the number of vertices which are from $\mathcal{S}_{2}$. It is easy to see in this case $n^{t_{1}}p^{t_{2}}= n^{t}\gamma^{t_{2}}\le n^{t}\left(\gamma \vee 1 \right)^{t}$. 
\end{proof}

\paragraph{Proof of part (i)}
We complete the proof of this part in two steps. 
In the first step we calculate the asymptotic variances of $(B_{n,k_1},\ldots, B_{n,k_l})$. 
The second step is dedicated towards proving the asymptotic normality and independence of $(B_{n,k_1},\ldots, B_{n,k_l})$.

\subparagraph{Step 1 (Calculation of variance):}
Under $H_0$, the case $k_{1}=1$ is simple as it is a sum of i.i.d.~random variables and hence its variance calculation is omitted.
One important thing to note is that the case $k=1$ depends on $\E[X_{i,j}^{4}]$ (which is equal to $3$ in the current case). This makes the asymptotic variance of $B_{n,1}$ equal to $2\gamma$,
which is not be the case in general.
% variance of $B_{n,1}$ is not $2\gamma$.

In what follows, we focus on the case when $k_{1} \ge 2$. 
Now we prove that $\Var(B_{n,k}) = (1+o(1)) 2k\gamma^{k}$ for any finite $k$. 
Define for any bi-word $w= (\alpha_{0},\beta_{0},\alpha_{1},\beta_{1},\ldots, \alpha_{k})$, 
\begin{equation}
\label{eq:X-w}
X_{w}:=  X_{\alpha_0,\beta_0}X_{\alpha_1,\beta_0} X_{\alpha_1,\beta_1} X_{\alpha_2,\beta_1} \ldots X_{\alpha_{k-1},\beta_{k-1}} X_{\alpha_0,\beta_{k-1}}-  \Indc_{k=1}.
\end{equation}
Now observe that 
\begin{equation}
	\label{eqn:varexpI}
\begin{aligned}
\mathrm{Var}(B_{n,k}) &= \left(\frac{1}{n}\right)^{2k}
\Expect\left[\left(\sum_{w} X_{w}\right)^2\right] = \left(\frac{1}{n}\right)^{2k} \Expect \left[ \sum_{w_{1},w_{2}}X_{w_{1}}X_{w_{2}} \right].
\end{aligned}
\end{equation}
Since both $X_{w_{1}}$ and $X_{w_{2}}$ are products of independent mean $0$ random variables that appears exactly once with $X_{w_1}$ or $X_{w_2}$, $\Expect[X_{w_{1}}X_{w_{2}}] \ne 0$ if and only if all the edges in $G_{w_{1}}$ are repeated in $G_{w_{2}}$. 
This happens only if $w_{1}$ and $w_{2}$ are paired. 
Now there are $(1+o(1))n^{k}p^{k}$ choices for $w_{1}$ and for each $w_{1}$ there are exactly $2k$ $w_{2}$'s such that $w_{1}$ and $w_{2}$ are paired (images of cyclic permutations of $w_1$ and of $\tilde{w}_1$). 
As a consequence, 
\[
\Var(B_{n,k}) = (1+o(1)) 2k\frac{n^{k}p^{k}}{n^{2k}}= (1+o(1)) 2k\gamma^{k}.
\]

\subparagraph{Step 2 (Proof of asymptotic normality):}  In order to complete this step, it suffices to prove the following two limits:
\begin{equation}\label{eqn:lim1}
\lim_{n \to \infty} \Expect\left[\left(B_{n,k_1}- p\Indc_{k_{1}=1}\right)B_{n,k_2} \right] \to 0
\end{equation}
whenever $k_1 < k_2$ and there exist random variables $Z_1,\ldots,Z_m$ such that for any fixed $m$
\begin{equation}\label{eqn:lim2}
\lim_{n \to \infty}\Expect[X_{n,1}\ldots X_{n,m}]\to\left\{
\begin{array}{ll}
  \sum_{\eta} \prod_{i=1}^{\frac{m}{2}} \Expect[Z_{\eta(i,1)}Z_{\eta(i,2)}] & ~\text{for $m$ even},\\
  0 &~\text{for $m$ odd.}
\end{array}
\right.
\end{equation}
where $X_{n,i} \in \left\{ \frac{B_{n,k_1}- p\Indc_{k_{1}=1}}{\sqrt{2k_1\gamma^{k_{1}}}},\ldots, \frac{B_{n,k_l}}{\sqrt{2k_l\gamma^{k_{l}}}} \right\}$.
To see this, observe that (\ref{eqn:lim2}) simultaneously imply parts $i)$ and $ii)$ of Lemma \ref{lem:mom}. 
The implication of part $i)$ is obvious. 
For part $ii)$ one can take $X_{n,i}$'s to be all equal and from Wick's formula (Lemma \ref{lem:wick}) the limiting distribution of $X_{n,i}$'s are normal and it is well known that normal random variables satisfy Carleman's condition. 
In addition, (\ref{eqn:lim2}) also implies that the limiting distribution of $\left(\frac{B_{n,k_1}- p\Indc_{k_{1}=1}}{\sqrt{2k_1\gamma^{k_{1}}}},\ldots, \frac{B_{n,k_l}}{\sqrt{2k_l\gamma^{k_{l}}}} \right)$ is multivariate normal.
Hence one gets the asymptotic independence by applying (\ref{eqn:lim1}).

\medskip

We first prove (\ref{eqn:lim1}). Observe that 
\[
\Expect\left[\left(B_{n,k_1}-p\Indc_{k_{1}=1}\right)B_{n,k_2} \right] = \left(\frac{1}{n}\right)^{{k_1+k_2}} \Expect \left[ \sum_{w_{1},w_{2}}X_{w_{1}}X_{w_{2}} \right].
\]
However, here $l(w_{1})\neq l(w_{2})$. So $\Expect \left[ X_{w_{1}}X_{w_{2}} \right]=0$. As a consequence, (\ref{eqn:lim1}) holds.

\medskip 
Now we prove (\ref{eqn:lim2}). 
Let $l_i-1$ be the length of the bipartite cycle corresponding to $X_{n,i}$ (so that $l_i$ is the length of the word corresponding to the bipartite cycle). 
Observe that $\frac{l_{i}-1}{2} \in \{k_1,\ldots, k_l\}$ for any $i$. 
At first we expand the left hand side of (\ref{eqn:lim2}) as
\begin{equation}
	\label{eqn:break}
\Expect[X_{n,1}\ldots X_{n,m}]=\left(\frac{1}{n}\right)^{\frac{1}{2}\sum_{i}(l_i-1)} \sum_{w_1,\ldots,w_m} \Expect\left[X_{w_1}\ldots X_{w_m}\right].
\end{equation}
Here each of the graphs $G_{w_1},\ldots,G_{w_m}$ are cycles of length $l_1-1,\ldots,l_m-1$ respectively. 
% \nb{note that the $l_i$'s are not the lengths of the words! each is one less than the length of the word}

In order to have $\Expect\left[X_{w_1}\ldots X_{w_m}\right] \neq 0$, we need each of the edges in $G_{w_1},\ldots, G_{w_m}$ to be traversed more than once. This is true even for $l_{i}=2$ for some $i$. 
In particular, in this case $G_{w_{i}}$ is a single edge and this edge is traversed twice. 
So one can think this as a cycle of length $2$.
Thus, $a = (w_1,\dots, w_m)$ is a weak CLT sentence.
Given any weak CLT sentence $a$, we introduce a partition $\eta(a)$ of $\{1,\ldots,m \}$ in the following way: If $i$ and $j$ are in same block of the partition $\eta(a)$, then $G_{w_i}$ and $G_{w_j}$ have at least one edge in common. 
As a consequence, we can further expand the right hand side of (\ref{eqn:break}) as
\begin{equation}
	\label{eqn:breakII}
\left(\frac{1}{n}\right)^{\frac{1}{2}\sum_{i}(l_i-1)}\sum_{\eta}\sum_{\substack{w_1,\ldots,w_m: \\
\eta(w_1,\ldots,w_m)=\eta} }\Expect\left[X_{w_1}\ldots X_{w_m}\right].
\end{equation}

We now show that we only need to care about those $\eta$'s which have at most $\lfloor \frac{m}{2} \rfloor$ blocks when evaluating the expectation. 
For any number $b$, $\lfloor b\rfloor$ denotes the largest integer that is no larger than $b$.
To this end, observe that each block in $\eta$ should have at least $2$ elements.
Otherwise there is an $i$ such that $G_{w_{i}}$ does not share any edge with $G_{w_{j}}$ for any $j \neq i$. 
Hence the random variables $X_{w_{i}}$ and $\prod_{j \neq i} X_{w_{j}}$ are independent, and so $\Expect\left[X_{w_1}\cdots X_{w_m}\right] = \Expect[X_{w_i}]\, \Expect[\prod_{j\neq i} X_{w_j}] =0$ from definition.
As a consequence, in order for $\Expect\left[X_{w_1}\cdots X_{w_m}\right] \neq 0$, the number of blocks in $\eta \le \lfloor \frac{m}{2} \rfloor$.

In what follows, we show that only those $\eta$'s such that the number of blocks in them are exactly $\frac{m}{2}$ contribute to a non-vanishing asymptotic mean. Note that this necessarily requires $m$ to be even.

When $\eta(w_1,\ldots,w_m)$ have strictly less than $\lfloor\frac{m}{2}\rfloor$ blocks (including all cases of odd $m$ and the case of even $m$ when the number of blocks is strictly less than $\frac{m}{2}$), $G_a$ has strictly less than $\lfloor \frac{m}{2} \rfloor$ connected components. 
From Lemma 4.10 of \cite{AZ05} it follows that in this case $\#V_{a}< \sum_{i=1}^{m}\frac{l_i-1}{2}$.
Applying Lemma \ref{lem:appendix} and noting that the $a$'s are weak CLT sentences, we have
% \nb{the following inequalities need double-checking!}
\begin{equation}
	\label{eqn:bddweakcltII}
\begin{split}
&\left(\frac{1}{n}\right)^{\frac{1}{2}\sum_{i}(l_i-1)}
\sum_{a:\#V_{a}< \sum_{i=1}^{m}\frac{l_i-1}{2}} 
\Expect\left[X_{w_1}\ldots X_{w_m}\right]\\
&\le \left(\frac{1}{n}\right)^{\frac{1}{2}\sum_{i}(l_i-1)} \sum_{t<\frac{1}{2}\sum_{i}(l_i-1)}\left(C_1\sum_{i}l_i\right)^{C_2m}\left(\sum_{i}l_i\right)^{3(\sum_{i}l_i-2t)}n^{t}\left(\gamma \vee 1 \right)^{t} \Expect\left[\left| X_{11} \right|^{\frac{1}{2}\sum_{i}l_i}\right]\\
& \le \Expect\left[\left| X_{11} \right|^{\frac{1}{2}\sum_{i}l_i}\right]\left(C_1\sum_{i}l_i\right)^{C_2m}\left(\sum_{i}l_i\right)^{3m} \left(\gamma \vee 1 \right)^{\frac{1}{2}\sum_{i}l_i} \sum_{t<\frac{1}{2}\sum_{i}(l_i-1)} \left(\frac{\left(\sum_{i}l_i\right)^{3}}{\sqrt{n}}\right)^{\sum_{i}(l_i-1)-2t}\\
& \le 
\left(C_{3}\sum_{i}l_{i}\right)^{C_{4}\sum_{i}l_{i}}
% \left(C_1\sum_{i}l_i\right)^{C_2m}
\left(\gamma \vee 1 \right)^{\frac{1}{2}\sum_{i}l_i}O\left(\frac{\left(\sum_{i}l_i\right)^{3}}{\sqrt{n}}\right)
% (1+O(1)).
% \frac{\left(\sum_{i}l_i\right)^{3}}{n}
\end{split}
\end{equation}
Here we have also used the fact for any standard Gaussian random variable $\Expect[|X|^{l}]\le (C_{3}l)^{C_{4}l}$. 
Observe that the rightmost side of (\ref{eqn:bddweakcltII}) is $o(1)$ 
% the dominant term is  $\left(C_{3}\sum_{i=1}^ml_{i}\right)^{C_{4}\sum_{i=1}^ml_{i}}$ and
since for $l_1,\dots, l_m = o(\sqrt{\log n})$, 
$\left(C_{3}\sum_{i=1}^ml_{i}\right)^{C_{4}\sum_{i=1}^ml_{i}}/n^\alpha \to 0$ whenever $\alpha > 0$ and $m$ is finite\footnote{In fact the term $\Expect[|X_{11} |^{{\sum_{i}l_i}/{2}}]$ is not optimal. One can prove the CLT under the null upto $o(\log n)$ order by the arguments similar to (2.1.32) in Anderson et al. \cite{AGZ}. However for our purpose this suffices.}.

Now the only remaining partitions are pair partitions which have exactly $\frac{m}{2}$ many blocks (and so naturally $m$ is even).
We now fix a partition $\eta$ of this kind. 
Let for any $i \in \{1,\ldots,\frac{m}{2}\}$, $\eta(i,1)<\eta(i,2)$ be the elements in the $i$th block. 
Observe now that fixing a pair partition $\eta$ and $(w_1,\ldots,w_{m})$ such that $\eta(w_{1},\ldots,w_{m})=\eta$, the random variables $X_{w_{\eta(i_1,j)}}$ and $X_{w_{\eta(i_2,j)}}$ are independent when ever $i_{1}\neq i_{2}$ for any $j\in \{ 1,2 \}$.
As a consequence, we now can rewrite (\ref{eqn:breakII}) as 
\begin{equation}\label{eqn:bal}
\begin{split}
&\left(\frac{1}{n}\right)^{\frac{1}{2}\sum_{i}(l_i-1)}
\sum_{\eta}\sum_{\substack{w_1,\ldots,w_m: \\
\eta(w_1,\ldots,w_m)=\eta}}\Expect\left[X_{w_1}\ldots X_{w_m}\right]\\
&=o(1)
+\left(\frac{1}{n}\right)^{\frac{1}{2}\sum_{i}(l_i-1)}
\sum_{\eta ~ \text{pair partition}}\sum_{\substack{w_1,\ldots,w_m: \\
\eta(w_1,\ldots,w_m)=\eta}} \prod_{i=1}^{\frac{m}{2}} \Expect\left[ X_{w_{\eta(i,1)}}X_{w_{\eta(i,2)}}\right]
\end{split}
\end{equation}
Now observe that whenever $\prod_{i=1}^{\frac{m}{2}} \Expect[ X_{w_{\eta(i,1)}}X_{w_{\eta(i,2)}}]\neq 0$, 
we have $w_{\eta(i,1)}$ and $w_{\eta(i,2)}$ are paired. 
When $l(w_{\eta(i,1)})=l(w_{\eta(i,2)})\neq 3$, 
there are $(1+o(1))(l_{\eta(i,1)}-1)(n\sqrt{\gamma})^{l_{\eta(i,1)}-1}$ many such choices of $(w_{\eta(i,1)}, w_{\eta(i,2)})$ for every $i$.
Here $l_{\eta(i,1)}-1$ equals the common length of the cycles induced by $w_{\eta(i,1)}$ and $w_{\eta(i,2)}$.
In this case 
$\Expect[ X_{w_{\eta(i,1)}}X_{w_{\eta(i,2)}}]= 1$. 
On the other hand, when $l(w_{\eta(i,1)})=3$, there are $(1+o(1))n^{l_{\eta(i,1)}-1} \gamma$ many such choices of $(w_{\eta(i,1)}, w_{\eta(i,2)})$ for every $i$ and in this case 
$\Expect [ X_{w_{\eta(i,1)}}X_{w_{\eta(i,2)}} ]= 2$.
Hence, we get the following further reduction of the right side of (\ref{eqn:bal}):
\begin{equation}
\begin{split}
% &\left(\frac{1}{n}\right)^{\frac{\sum_{i}l_i}{2}}\sum_{\eta}\sum_{w_1,\ldots,w_m~|~ \eta = \eta(w_1,\ldots,w_m)}E\left[X_{w_1}\ldots X_{w_m}\right]\\
&~~o(1) + (1+o(1))\left(\frac{1}{n}\right)^{\frac{1}{2}\sum_{i}(l_i-1)}
\sum_{\eta ~ \text{pair partition}}
% \sum_{\substack{w_1,\ldots,w_m: \\
% \eta(w_1,\ldots,w_m)=\eta}}
\prod_{i=1}^{\frac{m}{2}} (l_{\eta(i,1)}-1)\mathbf{1}_{l_{\eta(i,1)}= l_{\eta(i,2)}}(n\sqrt{\gamma})^{l_{\eta(i,1)}-1}
\\
&= o(1)+ (1+o(1))
\sum_{\eta ~ \text{pair partition}}
% \sum_{\eta~|~ \eta ~ \text{pair parition}}
\prod_{i=1}^{\frac{m}{2}} (l_{\eta(i,1)}-1)\gamma^{\frac{1}{2} (l_{\eta(i,1)}-1)}
\mathbf{1}_{l_{\eta(i,1)}= l_{\eta(i,2)}}.
\end{split}
\end{equation}
Recalling that $l_{i}=2k_{i}+1$ we complete the proof.
 \hfill $\square$

\paragraph{Proof of part (ii)}
We at first look at the case when $k=1$. This is an exceptional case and needs to be handled differently. 
Then we deal with the general case of $k\geq 2$.

\subparagraph{Analysis of $B_{n,1}$:}
Recall that $B_{n,1} = \frac{1}{n}\sum_{i=1}^{n}\sum_{j=1}^{p} X_{i,j}^{2}$. 
We have
\begin{equation}
% \begin{split}
B_{n,1}\left| \Theta, U  \right. 
=  \frac{1}{n}\sum_{i=1}^{n}\sum_{j=1}^{p} \left(  Z_{i,j}+ M_{i,j} \right)^{2}
% \end{split}
\end{equation}
where for any $(i,j)$,
\begin{equation}
	\label{eq:Mij}
M_{i,j} = \frac{1}{\sqrt{p}} \sum_{l=1}^{\kappa} \Theta_{i,l} U_{j,l}
\end{equation}
and $Z_{i,j} \stackrel{iid}{\sim} N(0,1)$. 
Observe that in this case one can apply the Lindeberg--Feller central limit theorem.
 % \nb{expand with a bit more detail?}.
So it suffices to calculate the limiting mean and variance of $B_{n,1}\left|\Theta, U\right.$. 
Now 
\begin{equation}
\begin{split}
\Expect\left[X_{i,j}^{2}\left|\Theta, U\right.\right] &= 
1 + M_{i,j}^2,
% \Var(Z_{i,j}) + \left[\mathrm{Bias}(X_{i,j})\left|\Theta, U\right. \right]^2\\
% & =1+ \left(\sum_{l=1}^{\kappa} \frac{1}{\sqrt{p}}\Theta_{i,l} U_{j,l}\right)^{2}
\end{split}
\end{equation}
and 
\begin{equation}
\begin{split}
\Var\left[ X_{i,j}^{2}| \Theta, U \right]& = 
% \Var\left[Z_{i,j}^{2} + 2Z_{i,j}M_{i,j} + M_{i,j}^{2} | \Theta, U\right]\\
% & =
\Var\left[Z_{i,j}^{2} + 2Z_{i,j}M_{i,j} | \Theta, U \right]\\
& = \Var\left[ Z_{i,j}^{2}\right] + 4 \Var[Z_{i,j}] M_{i,j}^{2}\\
& = 2 + 4 M_{i,j}^2.
\end{split}
\end{equation}
So it is enough to prove 
\begin{equation}\label{first_second}
\frac{1}{n} \sum_{i,j}M_{i,j}^{2} 
\stackrel{p}{\to} 
\sum_{l_{1},l_{2}} \Sigma_{\Theta}(l_{1},l_{2})\Sigma_{U}(l_{1},l_{2}).
\end{equation}
As a consequence,
\[
\Var\left[ B_{n,1}\right]= \frac{1}{n^2} \left(2np+ \sum_{i,j} 4M_{i,j}^2\right) \to 2\gamma.
\]
% Note that this also implies the desired variance formula since
% \begin{equation}
% \begin{split}
% \Var\left[B_{n,1}|\Theta, U\right]
% &= \frac{1}{n^{2}} \sum_{i,j} \Var\left[ X_{i,j}^{2}|\Theta, U\right]
% = \frac{1}{n^{2}} 2np + \frac{1}{n^2} 4 \sum_{i,j}M_{i,j}^{2},
% \end{split}
% \end{equation}
% and the second term on the right side goes to $0$ in probability from (\ref{first_second}).
To this end, note that
\begin{equation}
\begin{split}
\frac{1}{n} \sum_{i,j}M_{i,j}^{2}
 % \frac{1}{n}\left[\sum_{i,j}\left( \sum_{l=1}^{\kappa} \frac{1}{\sqrt{p}}\Theta_{i,l} U_{j,l} \right)^{2}\right]
 &= \frac{1}{n}\left[\sum_{i,j} \sum_{l,l'} \frac{1}{p}\Theta_{i,l}\Theta_{i,l'} U_{j,l} U_{j,l'}\right]\\
 &= \sum_{l=1}^{\kappa}\sum_{l'=1}^{\kappa}\left(\frac{1}{n} \sum_{i=1}^n\Theta_{i,l}\Theta_{i,l'} \right) \left(\frac{1}{p} \sum_{j=1}^p U_{j,l} U_{j,l'} \right).
\end{split}
\end{equation}
The weak law of large numbers then gives
\begin{equation*}
\frac{1}{n} \sum_{i=1}^n\Theta_{i,l}\Theta_{i,l'} 
\stackrel{p}{\to} \Sigma_{\Theta}(l,l') 
\qquad
\mbox{and}
\qquad
\frac{1}{p}\sum_{j=1}^p U_{j,l} U_{j,l'}  
\stackrel{p}{\to} \Sigma_{U}(l,l').
\end{equation*}
%  $$\frac{1}{n}\left( \sum_{i}\Theta_{i,l}\Theta_{i,l'} \right) \stackrel{p}{\to} \Sigma_{\Theta}(l,l')$$ and
% $$\frac{1}{p}\left( \sum_{j}U_{j,l} U_{j,l'} \right) \stackrel{p}{\to} \Sigma_{U}(l,l'). $$
Since $\kappa$ is fixed, we obtain (\ref{first_second}).

\subparagraph{Analysis of $B_{n,k}$ with $k \geq 2$:}
We first write 
\begin{equation}
\begin{split}
B_{n,k} & = \frac{1}{n^{k}} \sum_{i_0, j_0,  \ldots,i_{k-1}, j_{k-1}} X_{i_0,j_0}  \ldots  X_{i_0,j_{k-1}}\\ 
&= \frac{1}{n^{k}} \sum_{i_0, j_0,\ldots,i_{k-1}, j_{k-1}} \left(Z_{i_0,j_0}+M_{i_0,j_0} \right)  \ldots  \left(Z_{i_0,j_{k-1}}+M_{i_0,j_{k-1}}\right)\\
&= \frac{1}{n^{k}} \sum_{i_0, j_0,\ldots,i_{k-1}, j_{k-1}}Z_{i_0,j_0}  \ldots  Z_{i_0,j_{k-1}} 
+ \mu_{n,k} + V_{n,k},
% + \frac{1}{n^{k}} \left( \sum_{l=1}^{\kappa} \frac{1}{\sqrt{p}}\Theta_{i_0,l} U_{j_0,l} \right)\ldots\left( \sum_{l=1}^{\kappa} \frac{1}{\sqrt{p}}\Theta_{i_0,l} U_{j_{k-1},l} \right)+ V_{n,k}
\end{split}
\end{equation}
where 
\begin{equation}
	\label{eq:mu-k}
\mu_{n,k}:=\frac{1}{n^{k}}  \sum_{i_0, j_0,\ldots,i_{k-1}, j_{k-1}}
M_{i_0, j_0} \cdots M_{i_0, j_{k-1}},
% \left( \sum_{l=1}^{\kappa} \frac{1}{\sqrt{p}}\Theta_{i_0,l} U_{j_0,l} \right)
% \ldots
% \left( \sum_{l=1}^{\kappa} \frac{1}{\sqrt{p}}\Theta_{i_0,l} U_{j_{k-1},l} \right).
\end{equation}
and $V_{n,k}$ collects all the terms involving cross-products.
% Here $V_{n,k}$ is obtained by taking the sum of all the remaining terms in the expansion of
% \[
% \frac{1}{n^{k}} \sum_{i_0, j_0,\ldots,i_{k-1}, j_{k-1}} \left(Z_{i_0,j_0}+\sum_{l=1}^{\kappa} \frac{1}{\sqrt{p}}\Theta_{i_0,l} U_{j_0,l} \right)  \ldots  \left(Z_{i_0,j_{k-1}}+\sum_{l=1}^{\kappa} \frac{1}{\sqrt{p}}\Theta_{i_0,l} U_{j_{k-1},l}\right)
% \]
% apart from
% \[
% \frac{1}{n^{k}} \sum_{i_0, j_0,\ldots,i_{k-1}, j_{k-1}}Z_{i_0,j_0}  \ldots  Z_{i_0,j_{k-1}}
% \]
% and

The proof of the asymptotic normality of $\frac{1}{n^{k}} \sum_{i_0, j_0,\ldots,i_{k-1}, j_{k-1}}Z_{i_0,j_0}  \ldots  Z_{i_0,j_{k-1}}$ is the same as the proof we have just finished for the null distribution.
We shall prove later that $\mu_k$ satisfies \eqref{eq:mu-k-b}.
% \[
% \mu_{k} \stackrel{p}{\to} \sum_{l_{1}\ldots,l_{2k}\in \{1,\ldots,\kappa\}^{2k}} \Sigma_{\Theta}(l_{1},l_{2k})\Sigma_{U}(l_{1},l_{2})\Sigma_{\Theta}(l_{2},l_{3})\ldots \Sigma_{U}(l_{2k-1},l_{2k}).
% \]
Now we focus on $V_{n,k}$.
Observe that $\Expect\left[V_{n,k} \,|\, \Theta, U \right]=0$ and hence $\Expect[V_{n,k}]=0$. 
So our goal is to prove $\Expect[V_{n,k}^{2}] \to 0$ which implies $V_{n,k} \stackrel{p}{\to} 0$. 

Note that $V_{n,k}= \sum_{w} V_{n,k,w}$ where the summation is over all closed bi-words of length $2k+1$.
Fix such a bi-word $w$ and let $\emptyset \subsetneq E_{f}\subsetneq E_{w}$ be any subset. Then
\[
V_{n,k,w}= \frac{1}{n^{k}}\sum_{\emptyset \subsetneq E_{f}\subsetneq E_{w}} \mu(E_{f},w) \prod_{ e \in E_{w} \backslash E_{f}} Z_{e}.
\]
Here
\[
\mu(E_{f},w) = \prod_{e \in E_{f}} M_{\alpha_e,\beta_e}.
% \left( \sum_{l=1}^{\kappa} \frac{1}{\sqrt{p}}\Theta_{e_{1},l} U_{e_{2},l} \right)
\]
where for any edge $e$, $\alpha_e$ and $\beta_e$ denote its two end points which belong to $\mathcal{S}_{1}$ and $\mathcal{S}_{2}$ respectively.
Now 
\begin{equation}
\begin{split}
\Expect\left[V_{n,k}^{2}\left| \Theta, U \right.\right] &= \sum_{w_{1},w_{2}} \Expect\left[V_{n,k,w_{1}} V_{n,k,w_{2}}\left|\Theta, U \right. \right].
\end{split}
\end{equation}
We now give an upper bound to $\Expect\left[V_{n,k,w_{1}} V_{n,k,w_{2}} \right]$. 
At first fix any word $w_{1}$ and the set $\emptyset \subsetneq E_{f} \subsetneq E_{w_{1}}$ and consider all the words $w_{2}$ such that $E_{w_{1}}\cap E_{w_{2}}= E_{w_{1}} \backslash E_{f}$. As every edge in $G_{w_{1}}$ and $G_{w_{2}}$ appear exactly once within $G_{w_1}$ and $G_{w_2}$,
\begin{equation}
\begin{split}
&~~~~~~ \Expect[V_{n,k,w_1} V_{n,k,w_{2}} \left| \Theta, U \right. ]\\
& = \sum_{
% E_f\subset E'
E_{w_{1}}\backslash E' \subset E_{w_{1}} \backslash E_{f}
}\left(\frac{1}{n}\right)^{2k}\left[\mu(E',w_{1})\mu(E_{w_2}\backslash(E_{w_1}\backslash E'),w_{2})\right]\Expect\prod_{e \in E_{w}\backslash E'}(Z_{e})^{2}\\
& =\sum_{E_{w_{1}}\backslash E' \subset E_{w_{1}} \backslash E_{f}}\left(\frac{1}{n}\right)^{2k} 
 \left[\mu(E',w_{1})\mu(E_{w_2}\backslash(E_{w_1}\backslash E'),w_{2})\right].
\end{split}
\end{equation}
Now it is enough to prove 
\begin{equation}\label{condsecvnw}
\begin{split}
&\Expect\left[\left(\frac{1}{n}\right)^{2k}\sum_{w_{1}}\sum_{\emptyset \subsetneq E_{f}\subsetneq E_{w}} \sum_{E_{f} \subset E'} \sum_{\{w_{2}| E_{w_{1}}\cap E_{w_{2}}= E_{w_{1}} \backslash E_{f} \} }\left[\mu(E',w_{1})\mu(E_{w_2}\backslash(E_{w_1}\backslash E'),w_{2})\right]\right]\\
& \le \left(\frac{1}{n}\right)^{2k}\sum_{w_{1}}\sum_{\emptyset \subsetneq E_{f}\subsetneq E_{w}} \sum_{E_{f} \subset E'} \sum_{\{w_{2} | E_{w_{1}}\cap E_{w_{2}}= E_{w_{1}} \backslash E_{f} \} }\Expect\left|\mu(E',w_{1})\mu(E_{w_2}\backslash(E_{w_1}\backslash E'),w_{2})\right| \to 0.
\end{split}
\end{equation}
Now observe that for any $w$ in consideration and any subset $E $ of $E_{w}$,
\begin{equation*}
\left| \mu(E,w)\right| = \left(\frac{1}{p}\right)^{\frac{\#E}{2}}\prod_{e \in E} 
% M_{\alpha_e,\beta_e}
\left| \sum_{l=1}^{\kappa} \Theta_{\alpha_e,l} U_{\beta_e,l} 
\right|.
\end{equation*}
Hence we have for any $E \subset E_{w_1}$ and $\bar{E}\subset E_{w_2}$ such that $\#E = \#\bar{E}$,
% \nb{the following equations need to be rewritten to accommodate different $E$'s for $w_1$ and $w_2$}
\begin{equation}
	\label{mue'mateI}
\begin{split}
% &\left| \mu(E,w)\right| \le \left|\left(\frac{1}{p}\right)^{\frac{\#E}{2}}\prod_{e \in E}\left( \sum_{l=1}^{\kappa} \Theta_{e_{1},l} U_{e_{2},l} \right)\right|\\
% & \le 2^{2k}\left(\frac{1}{p}\right)^{\frac{\#E}{2}} \left| \prod_{e \in E}\left( \sum_{l=1}^{\kappa} \Theta_{e_{1},l} U_{e_{2},l} \right) \right| \\
& \Expect\left| \mu(E,w_{1})\mu(\bar{E},w_{2})\right|\\
& \le \left(\frac{1}{p}\right)^{{\#E}} \prod_{e \in E} \Expect\left[ \left| \sum_{l=1}^{\kappa} \Theta_{\alpha_e,l} U_{\beta_e,l} \right|^{2\#E} \right]^{\frac{1}{2\# E}}  
\prod_{\bar{e} \in \bar{E}}\Expect\left[ \left| \sum_{l=1}^{\kappa} \Theta_{\alpha_{\bar{e}},l} U_{\beta_{\bar{e}},l} \right|^{2\#E} \right]^{\frac{1}{2\# E}}\\
& \le \left(\frac{1}{p}\right)^{\#E} \left( C_{5} \# E \right)^{C_{6} \# E}
\le \left(\frac{1}{p}\right)^{{\#E}} (C_{7}k)^{C_{8}k}.
\end{split}
\end{equation}
The last step follows from the fact no matter what the value of $e$ is, $\sum_{l=1}^{\kappa} \Theta_{\alpha_e,l} U_{\beta_e,l}$ is sub-exponential with parameter $C$ for some constant $C$ that depends on $\kappa$, $\widetilde{\Sigma}_\Theta$ and $\widetilde{\Sigma}_U$. 
Plugging  the estimate obtained in (\ref{mue'mateI}) in (\ref{condsecvnw}), we have 
\begin{equation}\label{mue'mateII}
\begin{split}
&\Expect\left[\left(\frac{1}{n}\right)^{2k}\sum_{w_{1}}\sum_{\emptyset \subsetneq E_{f}\subsetneq E_{w}} \sum_{E_{f} \subset E'} \sum_{\left\{w_{2}| E_{w_{1}}\cap E_{w_{2}}= E_{w_{1}} \backslash E_{f}\right\} }\left[\mu(E',w_{1})\mu(E_{w_2}\backslash(E_{w_1}\backslash E'),w_{2})\right]\right]\\
& \le \left(\frac{1}{n}\right)^{2k}2^{2k}\sum_{w_{1}}\sum_{\emptyset \subsetneq E_{f}\subsetneq E_{w}} \sum_{E_{f} \subset E'} \sum_{\left\{w_{2}| E_{w_{1}}\cap E_{w_{2}}= E_{w_{1}} \backslash E_{f}\right\} } \left(\frac{1}{p}\right)^{{\#E'}} (C_{7}k)^{C_{8}k}\\
& \le \left(\frac{1}{n}\right)^{2k}(C_{7}k)^{C_{8}k}\sum_{w_{1}}\sum_{\emptyset \subsetneq E_{f}\subsetneq E_{w}}  \left( \frac{1}{p} \right)^{\# E_{f}} \sum_{E_{f} \subset E'} \sum_{\left\{w_{2}|E_{w_{1}}\cap E_{w_{2}}= E_{w_{1}} \backslash E_{f}\right\} } 1.
\end{split}
\end{equation}
Observe that the graph corresponding to the edges $E_{w_{1}} \backslash E_{f}$ is a disjoint collection of line segments. 
Let the number of such line segments be $\zeta$. Obviously $\zeta \le \#(E_{w_{1}} \backslash E_{f})$.  
The number of ways these $\zeta$ components can be placed in $w_{2}$ is bounded by $(2k)^{\zeta}\le (2k)^{\#(E_{w_{1}} \backslash E_{f})}\le (2k)^{2k}$ and all other nodes in $w_{2}$ can be chosen freely. 
So there are at most $(1+o(1))\left[ \left(\gamma \vee 1\right) n \right]^{2k-\#V_{E_{w_{1}}\backslash E_f}}(2k)^{2k}$ choices of such $w_{2}$. Here $V_{E_{w}\backslash E_{f}}$ is the set of vertices of the graph corresponding to 
$G_w$ with all edges in $E_f$ removed, i.e., $E_w\backslash E_f$.
% $(E_{w} \backslash E_{f})$.
Observe that, whenever $2k = \E_w >\#E_{f}>0$, $E_{w}\backslash E_{f}$ is a forest and so $\#V_{E_{w}\backslash E_{f}}\ge \#(E_{w} \backslash E_{f})+1$ which is equivalent to
$$
2k-\#V_{E_{w}\backslash E_f} \le \#E_{f}-1.
$$ 
Also observe that there are no more than $2^{2k}$ many choices of $E_{f}$'s and for each $E_{f}$ there are no more than $2^{2k}$ many choices for $E'$'s. 
Combining all these, we have the rightmost side of (\ref{mue'mateII}) is bounded by 
\begin{equation}
\begin{split}
&\left(\frac{1}{n}\right)^{2k}(C_{7}k)^{C_{8}k}\sum_{w_{1}}\sum_{\emptyset \subsetneq E_{f}\subsetneq E_{w}}  \left( \frac{1}{p} \right)^{\# E_{f}} (2)^{2k} \times (2k)^{2k} \left[ \left(\gamma \vee 1\right) n \right]^{\#E_{f}-1}\\
&\le 
% \left(\frac{1}{n}\right)^{2k}
\frac{1}{p}(C_{7}k)^{C_{8}k} (2k)^{2k} 2^{4k} \left[ \frac{\gamma \vee 1}{\gamma} \right]^{2k} \to 0.
\end{split}
\end{equation}

\medskip

Now our final task is to prove $\mu_{n,k}\stackrel{p}{\to}\mu_k$ defined in \eqref{eq:mu-k-b}.
% \[
% \mu_{k} \stackrel{p}{\to} \sum_{l_{1}\ldots,l_{2k}\in \{1,\ldots,\kappa\}^{2k}} \Sigma_{\Theta}(l_{1},l_{2k})\Sigma_{U}(l_{1},l_{2})\Sigma_{\Theta}(l_{2},l_{3})\ldots \Sigma_{U}(l_{2k-1},l_{2k}).
% \]
First we expand $\mu_{n,k}$ in \eqref{eq:mu-k} as 
\begin{equation}
\begin{split}
\mu_{n,k}  
% &= \frac{1}{n^{k}} \sum_{i_0, j_0,\ldots,i_{k-1}, j_{k-1}}\left( \sum_{l=1}^{\kappa} \frac{1}{\sqrt{p}}\Theta_{i_0,l} U_{j_0,l} \right)\ldots\left( \sum_{l=1}^{\kappa} \frac{1}{\sqrt{p}}\Theta_{i_0,l} U_{j_{k-1},l} \right)\\
&= \frac{1}{n^{k}} \frac{1}{p^{k}} \sum_{i_0, j_0,\ldots,i_{k-1}, j_{k-1}} \sum_{l_{1},\ldots,l_{2k}} \Theta_{i_0,l_{1}} U_{j_0,l_{1}}\ldots \Theta_{i_0,l_{2k}} U_{j_{k-1},l_{2k}}\\
&= \sum_{l_{1},\ldots,l_{2k}} \left(  \frac{1}{n^{k}}\left[ \sum_{i_{0},\ldots, i_{k-1}}\Theta_{i_0,l_{1}} \Theta_{i_0,l_{2k}} \Theta_{i_1,l_{2}} \Theta_{i_{1},l_{3}} \ldots \Theta_{i_{k-1},l_{2k-2}} \Theta_{i_{k-1},l_{2k-1}}\right] \right. \times \\
&~~~~~~~~~~~~~~~~ \left. \frac{1}{p^{k}} \left[\sum_{j_{0},\ldots, j_{k-1}} U_{j_{0},l_{1}}U_{j_0,l_{2}} U_{j_1,l_{3}} U_{j_{1},l_{4}} \ldots U_{j_{k-1},l_{2k-1}} U_{j_{k-1},l_{2k}} \right] \right).
\end{split}
\end{equation}
Now fix the values of $l_{1},\ldots,l_{2k}$ and for this value of the group assignment we have
\begin{align*}
\Expect\left[\frac{1}{n^{k}}\sum_{i_{0},\ldots, i_{k-1}}\Theta_{i_0,l_{1}} \Theta_{i_0,l_{2k}} \Theta_{i_1,l_{2}} \Theta_{i_{1},l_{3}} \ldots \Theta_{i_{k-1},l_{2k-2}} \Theta_{i_{k-1},l_{2k-1}}\right]~~~~~\\
= m^\Theta_{l_1,\dots, l_{2k}} = (1+o(1))\Sigma_{\Theta}(l_{1},l_{2k})\ldots \Sigma_{\Theta}(l_{2k-2},l_{2k-1})&.	
\end{align*}
Now
\begin{equation}\label{thetavariancee'mateI}
\begin{split}
&\Var\left[\frac{1}{n^{k}}\sum_{i_{0},\ldots, i_{k-1}}\Theta_{i_0,l_{1}} \Theta_{i_0,l_{2k}} \Theta_{i_1,l_{2}} \Theta_{i_{1},l_{3}} \ldots \Theta_{i_{k-1},l_{2k-2}} \Theta_{i_{k-1},l_{2k-1}}  \right]\\
&= \frac{1}{n^{2k}}\sum_{i_{0}^{(1)},\ldots, i_{k-1}^{(1)}}\sum_{i_{0}^{(2)},\ldots, i_{k-1}^{(2)}} \Expect\left[ \left(\Theta_{i_0^{(1)},l_{1}}\Theta_{i_{0}^{(1)},l_{2k}} \ldots \Theta_{i_{k-1}^{(1)},l_{2k-2}}\Theta_{i_{k-1}^{(1)},l_{2k-1}} - m^\Theta_{l_1,\dots, l_{2k}}\right)\times \right.\\
&~~~~~~~~~~~~~~~~~~~~~~~~~~~~~~~~~~~~~\left. \left(\Theta_{i_0^{(1)},l_{2}}\Theta_{i_{0}^{(1)},l_{2k}} \ldots \Theta_{i_{k-1}^{(2)},l_{2k-2}}\Theta_{i_{k-1}^{(2)},l_{2k-1}} - m^\Theta_{l_1,\dots, l_{2k}}\right)  \right].
\end{split}
\end{equation}
However, if the indices $(i_{0}^{(1)},\ldots, i_{k-1}^{(1)})$ and $(i_{0}^{(2)},\ldots, i_{k-1}^{(2)})$ are disjoint, 
\begin{equation*}
\begin{split}
&\Expect\left[ \left(\Theta_{i_0^{(1)},l_{1}}\Theta_{i_{0}^{(1)},l_{2k}} \ldots \Theta_{i_{k-1}^{(1)},l_{2k-2}}\Theta_{i_{k-1}^{(1)},l_{2k-1}} - m^\Theta_{l_1,\dots, l_{2k}}\right)\times \right.\\
&~~~~~~\left. \left(\Theta_{i_0^{(1)},l_{2}}\Theta_{i_{0}^{(1)},l_{2k}} \ldots \Theta_{i_{k-1}^{(2)},l_{2k-2}}\Theta_{i_{k-1}^{(2)},l_{2k-1}} - m^\Theta_{l_1,\dots, l_{2k}}\right)  \right] =0 .
\end{split}
\end{equation*}
Now consider the indices 
% \nb{need to convert to $\#S$ or $|S|$ throughout}
$$\mathcal{A}:=\left\{(i_{0}^{(1)},\ldots, i_{k-1}^{(1)}),(i_{0}^{(2)},\ldots, i_{k-1}^{(2)})~\left|\right.~ 
\#\big(\{ i_{0}^{(1)},\ldots, i_{k-1}^{(1)} \} \cap \{i_{0}^{(2)},\ldots, i_{k-1}^{(2)} \} \big)\ge 1 \right\}.$$
It is easy to see $\#\mathcal{A}\le \left(c_{1}k\right)^{c_{2}k} n^{2k-1}$. Further from sub-Gaussianity and H\"{o}lder's inequality we also have
 \begin{equation*}
\begin{split}
&\Expect\left[\left| \left(\Theta_{i_0^{(1)},l_{1}}\Theta_{i_{0}^{(1)},l_{2k}} \ldots \Theta_{i_{k-1}^{(1)},l_{2k-2}}\Theta_{i_{k-1}^{(1)},l_{2k-1}} - m^\Theta_{l_1,\dots, l_{2k}}\right)\times \right.\right.\\
&~~~~~~\left. \left.\left(\Theta_{i_0^{(1)},l_{2}}\Theta_{i_{0}^{(1)},l_{2k}} \ldots \Theta_{i_{k-1}^{(2)},l_{2k-2}}\Theta_{i_{k-1}^{(2)},l_{2k-1}} - m^\Theta_{l_1,\dots, l_{2k}}\right)\right|  \right] = (c_{3}k)^{c_{4}k}
\end{split}
\end{equation*}
 uniformly over the indices. This gives us the final expression of (\ref{thetavariancee'mateI}) to be bounded by  $\frac{\left(c_{1}c_{3}k\right)^{(c_{2}+c_{4})k}}{n} \to 0.$
The  proof for 
\[
\frac{1}{p^{k}} \sum_{j_{0},\ldots, j_{k-1}} U_{j_{0},l_{1}}U_{j_0,l_{2}} U_{j_1,l_{3}} U_{j_{1},l_{4}} \ldots U_{j_{k-1},l_{2k-1}} U_{j_{k-1},l_{2k}} \stackrel{p}{\to} \Sigma_{U}(l_{1},l_{2})\Sigma_{U}(l_{3},l_{4})\ldots \Sigma_{U}(l_{2k-1},2k)
\]
is analogous and so we omit the details.
% A similar calculation shows that 
%\[
%\Indc_{\Omega_{n}}\frac{1}{p^{k}} \left[\sum_{j_{0},\ldots, j_{k-1}} U_{j_0,l_{0}}^{2}  U_{j_{1},l_{0}}^{2} \ldots U_{j_{k-1},l_{0}}^{2}  \right] \stackrel{p}{\to} 1.
%\]
% When $l_{1},\ldots,l_{2k}$ are not all the same, there exists $m$ such that $l_{m} \neq l_{m+1}$. If $m$ is even then $\Theta_{i_{\frac{m}{2}},l_{m}}$ and $\Theta_{i_{\frac{m}{2}},l_{m+1}}$ are independent mean zero random variables. On the other hand, if $m$ is odd then $\wt{U}_{j_{\frac{m-1}{2}},l_{m}}$ and $\wt{U}_{j_{\frac{m-1}{2}},l_{m+1}}$ are independent mean zero random variables. It is easy to show that in this case, 
%\begin{equation}
%\begin{split}
%&\E\left( \Indc_{\Omega_{n}} \frac{1}{n^{k}}\left[ \sum_{i_{0},\ldots, i_{k-1}}\Theta_{i_0,l_{1}} \Theta_{i_0,l_{2k}} \Theta_{i_1,l_{2}} \Theta_{i_{1},l_{3}} \ldots \Theta_{i_{k-1},l_{2k-2}} \Theta_{i_{k-1},l_{2k-1}}\right] \right. \times \\
%                           &~~~~~~~~~~~ \left. \frac{1}{p^{k}} \left[\sum_{j_{0},\ldots, j_{k-1}} U_{j_{0},l_{1}}U_{j_0,l_{2}} U_{j_1,l_{3}} U_{j_{1},l_{4}} \ldots U_{j_{k-1},l_{2k-1}} U_{j_{k-1},l_{2k}} \right] \right) \to 0.
%\end{split}
%\end{equation}
%Finally the variance calculation is similar to the case shown. We omit the details. Finally from CLT of $\frac{1}{\sqrt{p}}\inn{\wt{U}_{l}}{\wt{U}_{l'}}$, one can choose $\delta_{n}= \frac{1}{n^{\nu}}$ for some $0<\nu < \frac{1}{2}$. Hence $(1+ O(\delta_{n}))^{k} \to 1$ for $k= o(\sqrt{\log(n)})$. This completes the proof. 
\hfill{$\square$}

%% file: result-proof.tex
\section{Proof of main results}
\label{sec:proof}
% With Propositions \ref{prop:bicyclelim} and \ref{prop:norcont} in hand, we are now ready to prove Theorem \ref{thm:main}.\\

In this section, we focus on the proof of Theorem \ref{thm:main}. 
The proof of Theorem \ref{thm:mainvar} can be established analogously using the same strategy mentioned at the beginning of Section \ref{subsec:proof-bicycle}.

Throughout the proof, 
without further specification,
% when there is no additional information,
all probability and expectation calculations are conducted with respect to $\mathbb{P}_{0,n}$, i.e., under the null hypothesis.
For any two matrices $A=(a_{i,j})\in \reals^{{m_1 \times m_2}}$ and $B=(b_{i,j}) \in \reals^{n_1 \times n_2}$, we define their Kronecker product $A \otimes B$ as
\[
A \otimes B = \left(
\begin{array}{llll}
a_{1,1}B & a_{1,2}B& \ldots & a_{1,m_2}B\\
a_{2,1}B & a_{2,2}B & \ldots & a_{2,m_2}B\\
\vdots  & \vdots & \ldots  & \vdots\\
a_{m_1,1}B & a_{m_1,2}B & \ldots & a_{m_1,m_2}B
\end{array}
\right).
\]
In addition, $\mathrm{vec}(A) = (A_{*1}',\dots, A_{*m_2}')' \in \reals^{m_1m_2\times 1}$ is the vector obtained from stacking all column vectors of $A$ in order.

\subsection{Proof of parts 1 and 2}
Recall that $p = p_n$ is a sequence depending on $n$.
In this proof we shall use the following two sequences of $\sigma$-fields:
\begin{equation}
\mathcal{G}_{n} = \sigma\left( \{X_{i}\}_{i=1}^n \right),\qquad
\mathcal{F}_{n} = \sigma\left( \{\Theta_{i*}\}_{i=1}^n, \{U_{j*}\}_{j=1}^p \right).
\end{equation}
It is straightforward to verify that 
\begin{equation}
\begin{split}
L_{n}= \Expect [L_{n}^{f} | \mathcal{G}_{n}]
\end{split}
\end{equation}
where the expectation is taken over $\Theta$ and $U$ and for $M_{i,j}$ defined in \eqref{eq:Mij},
\[
L_{n}^{f}:= \exp \left\{ \sum_{i=1}^{n}\sum_{j=1}^{p} \left(X_{i,j}
M_{i,j}
% \left( \sum_{l=1}^{\kappa}\frac{1}{\sqrt{p}}\theta_{i,l}\mathbf{u}_{j,l}\right)
 - \frac{1}{2} M_{i,j}^2
 % \left(  \sum_{l=1}^{\kappa}\frac{1}{\sqrt{p}}\theta_{i,l}\mathbf{u}_{j,l} \right)^2
 \right)\right\}.
\]
 
\paragraph{Step 1.}
We now consider any sequence of events $\Omega_{n} \in \mathcal{F}_{n}$ such that $\Prob\left[ \Omega_{n}^{c} \right] \to 0$ as $n\to\infty$. 
An explicit description of the $\Omega_{n}$'s of our interest will be given in step 2.
Now define 
\[
\wt{L}_{n}:= \Expect \big[L_{n}^{f}\Indc_{\Omega_{n}}\,|\,\mathcal{G}_{n} \big].
\]
In the rest of this step, we argue that it suffices to prove the desired results for $\wt{L}_n$.
Since $\wt{L}_{n}\le L_{n}$ almost surely under $\mathbb{P}_{0,n}$, the measure $\widetilde{\mathbb{Q}}_{n}$ on $\mathcal{G}_{n}$ defined as 
\[
\widetilde{\mathbb{Q}}_{n}(A_{n})= 
\frac{1}{\Prob[ \Omega_{n}]}
\Expect_{\mathbb{P}_{0,n}}\big[\wt{L}_{n}\Indc_{A_n}\big], \qquad \forall ~ A_{n} \in \mathcal{G}_{n},
\]
is a probability measure. 
By definition,
\begin{equation}\label{unif_meas}
\begin{split}
0 & \le \left|\mathbb{P}_{1,n}(A_{n})- \widetilde{\mathbb{Q}}_{n}(A_{n})\right|
\\
&\le \frac{1}{\Prob\left[ \Omega_{n} \right]}
\Expect_{\mathbb{P}_{0,n}}[ ( L_{n}- \wt{L}_{n} )\Indc_{A_n} ]
+ \mathbb{P}_{1,n}(A_{n})\frac{\Prob[ \Omega_{n}^{c} ]}{\Prob[ \Omega_{n}]}\\
& \le \frac{1}{\Prob\left[ \Omega_{n} \right]}
\Expect_{\mathbb{P}_{0,n}}\big[ L_{n}-\wt{L}_{n} \big]
+ \frac{\Prob\left[ \Omega_{n}^{c} \right]}{\Prob\left[ \Omega_{n}\right]}
% \\
% &
= \frac{1}{\Prob\left[ \Omega_{n} \right]}
\Expect\big[L_{n}^{f} \Indc_{\Omega_{n}^{c}} \big]
+\frac{\Prob\left[ \Omega_{n}^{c} \right]}{\Prob\left[ \Omega_{n}\right]}\\
& = \frac{1}{\Prob\left[ \Omega_{n} \right]}
\Expect\Big[ \Indc_{\Omega_{n}^{c}} \Expect\big[ L_{n}^{f} |\mathcal{F}_{n}\big] \Big]+\frac{\Prob\left[ \Omega_{n}^{c} \right]}{\Prob\left[ \Omega_{n}\right]}
= 2\,\frac{\Prob\left[ \Omega_{n}^{c} \right]}{\Prob\left[ \Omega_{n}\right]}\,.
\end{split}
\end{equation}
In other words, the total variation distance between $\bbP_{1,n}$ and $\widetilde{\bbQ}_n$ converges to zero.
% we have a uniform upper bound on $|\mathbb{P}_{1,n}(A_{n})- \widetilde{\mathbb{Q}}_{n}(A_{n})|$ for all $A_{n}\in \mathcal{G}_{n}$.
As a consequence, for any fixed $l\in \naturals$ and any $1\leq k_1 < \cdots < k_l = o(\sqrt{\log(n)})$, under $\widetilde{\bbQ}_n$,
\[
\left(\frac{B_{n,k_1}- p\Indc_{k_{1}=1}-\mu_{k_1}}{\sqrt{2k_1\gamma^{k_1}}},\ldots, \frac{B_{n,k_l}-\mu_{k_l}}{\sqrt{2k_l\gamma^{k_l}}}\right) \stackrel{d}{\to} N_{l}(0,I_{l}).
\]
Now if one can choose $\Omega_{n}$ in such a way that 
\begin{equation}
	\label{eq:l2cond}
	\begin{split}
\limsup_{n\to\infty} \Expect_{\mathbb{P}_{0,n}}\left[\wt{L}_{n}^{2} \right]
= 
\limsup_{n\to\infty} \Expect_{\mathbb{P}_{0,n}}\left[ \left(\frac{1}{\Prob\left[ \Omega_{n} \right]}\wt{L}_{n}\right)^{2} \right] \le \exp \sth{\sum_{k=1}^{\infty} \frac{\mu_{k}^{2}}{2k\gamma^{k}}} ,
% \Leftrightarrow & \limsup_{n\to\infty} \Expect_{\mathbb{P}_{0,n}}\left[\wt{L}_{n}^{2} \right] \le \exp \sth{\sum_{k=1}^{\infty} \frac{\mu_{k}^{2}}{2k\gamma^{k}}},
\end{split}
\end{equation}
then one can use Proposition \ref{prop:norcont} to conclude that 
\[
\left.\frac{1}{\Prob\left[ \Omega_{n} \right]}\wt{L}_{n}\right|{\mathbb{P}_{0,n}} \stackrel{d}{\to} \exp\left\{ \sum_{k=1}^{\infty} \frac{2\mu_{k}Z_{k}-\mu_{k}^{2}}{4k\gamma^{k}} \right\}.
\]
Hence, $\wt{L}_{n}\,|\,{\mathbb{P}_{0,n}}$ converges in distribution to the same limit.
% \[
% \wt{L}_{n}|{\mathbb{P}_{0,n}} \stackrel{d}{\to} \exp\left\{ \sum_{k=1}^{\infty} \frac{2\mu_{k}Z_{k}-\mu_{k}^{2}}{4k\gamma^{k}} \right\}.
% \]
Then it remains to prove that 
\[
L_{n}-\wt{L}_{n}\,|\,\mathbb{P}_{0,n} \stackrel{p}{\to} 0.
\] 
Observe that $L_{n} \ge \wt{L}_{n}$ almost surely under $\bbP_{0,n}$.
 % ($\mathbb{P}_{0,n}~ a.s.$)
If the last display is not true, then there exist positive constants $c_1,c_2>0$ and a subsequence $n_{k}$ such that 
\[
\liminf_{n_{k}\to\infty} \mathbb{P}_{0,n_{k}}\left[ \Indc_{L_{n_{k}}-\wt{L}_{n_{k}}>c_1} \right] \ge c_2.
\]
However
\begin{equation}
 \begin{split}
 &
\liminf_{n_{k}\to\infty}\left[\mathbb{P}_{1,n_{k}}\left( \Indc_{L_{n_{k}}-\wt{L}_{n_{k}}>c_1} \right)- \widetilde{\mathbb{Q}}_{n_{k}} \left( \Indc_{L_{n_{k}}-\wt{L}_{n_{k}}>c_1} \right)\right]\\
&\ge 
\liminf_{n_{k}\to\infty} \left\{c_{1} \mathbb{P}_{0,n_{k}}\left( \Indc_{L_{n_{k}}-\wt{L}_{n_{k}}>c_1} \right) - {\Prob\left[ \Omega_{n_{k}}^{c} \right]}\widetilde{\mathbb{Q}}_{n_{k}} \left( \Indc_{L_{n_{k}}-\wt{L}_{n_{k}}>c_1}\right)\right\} \ge c_{1}c_{2}.
 \end{split}
\end{equation}
This contradicts (\ref{unif_meas}) since that bound is uniform over all $A_{n} \in \mathcal{G}_{n}$.

\paragraph{Step 2.}
Now we prove \eqref{eq:l2cond} by making appropriate choices of the $\Omega_n$'s.
First observe that 
\begin{equation}
\begin{split}
\Expect
% _{\mathbb{P}_{0,n}}
\big[ \wt{L}_{n}^{2} \big]
&= \Expect
% _{\mathbb{P}_{0,n}}
\Big[ \Expect\big[L_{n}^{f}\Indc_{\Omega_n} \,|\, \mathcal{G}_{n}\big]^{2} \Big]
\\
&= \Expect
% _{\mathbb{P}_{0,n}}
\Big[ \Expect\big[L_{n}^{f(1)}L_{n}^{f(2)}\Indc_{\Omega_n^{(1)}} \Indc_{\Omega_n^{(2)}}\,\big|\,\mathcal{G}_{n}\big] \Big]\\
&= \Expect\Big[ L_{n}^{f(1)}L_{n}^{f(2)}\Indc_{\Omega_n^{(1)}} \Indc_{\Omega_n^{(2)}} \Big]\\
&= \Expect\Big[ \Indc_{\Omega_n^{(1)}} \Indc_{\Omega_n^{(2)}} 
\Expect\big[ L_{n}^{f(1)}L_{n}^{f(2)} \,\big|\, \mathcal{F}_{n} \big] \Big].
\end{split}
\end{equation} 
Here $L_{n}^{f(1)}$ and $L_{n}^{f(2)}$ are two independent copies of $L_{n}^{f}$ where the $X_{i}$'s are kept fixed but one takes two i.i.d.~copies of the $\Theta$'s and $U$'s. 
This is feasible (only) under the null hypothesis when the $X_i$'s are independent of $\Theta$ and $U$. 
With slight abuse of notation, we use $\calF_n$ to denote the $\sigma$-field generated by both copies.
We call the corresponding random variables 
$\{\Theta^{(1)}, U^{(1)}\}$ and 
$\{\Theta^{(2)}, U^{(2)}\}$.
Observe that 
\begin{equation}\label{secmom:exp}
\begin{aligned}
&\Expect\left[ L_{n}^{f(1)}L_{n}^{f(2)}\left| \mathcal{F}_{n}\right . \right]\\
& = 
\exp\left[ \sum_{i=1}^{n}\sum_{j=1}^{p}\left( \sum_{l=1}^{\kappa} \frac{1}{\sqrt{p}}\Theta_{i,l}^{(1)} U_{j,l}^{(1)} \right)
\left( \sum_{l=1}^{\kappa} \frac{1}{\sqrt{p}}\Theta_{i,l}^{(2)} U_{j,l}^{(2)}  \right) \right]\\
&= \exp \left[ \sum_{l_{1}=1}^{\kappa}\sum_{l_2=1}^{\kappa} \frac{1}{p}\inn{\Theta_{*l_1}^{(1)}}{\Theta_{*l_2}^{(2)}} \inn{U_{*l_1}^{(1)}}{U_{*l_2}^{(2)}}
\right].
% &= \exp \left[ \frac{1}{p}
% \left( \sum_{i=1}^{n} \theta_{i,1}^{(1)}\theta_{i,1}^{(2)} \right)
% \left( \sum_{j=1}^{p} \mathbf{u}_{j,1}^{(1)}\mathbf{u}_{j,1}^{(2)} \right)+ \frac{1}{\sqrt{p}\|\bla^{(2)}\|}
% \left( \sum_{i=1}^{n} \theta_{i,1}^{(1)}\theta_{i,2}^{(2)} \right)\left( \sum_{j=1}^{p} \mathbf{u}_{j,1}^{(1)}\mathbf{u}_{j,2}^{(2)} \right)+ \right.\\
%  &~~~~~~~~~~~~~\left.
%  \frac{1}{\sqrt{p}\|\bla^{(1)}\|}\left( \sum_{i=1}^{n} \theta_{i,2}^{(1)}\theta_{i,1}^{(2)} \right)\left( \sum_{j=1}^{p} \mathbf{u}_{j,1}^{(2)} \mathbf{u}_{j,2}^{(1)} \right)
%  + \frac{1}{\|\bla^{(1)}\| \|\bla^{(2)}\|}\left( \sum_{i=1}^{n} \theta_{i,2}^{(1)}\theta_{i,2}^{(2)} \right)\left( \sum_{j=1}^{p} \mathbf{u}_{j,2}^{(1)} \mathbf{u}_{j,2}^{(2)}\right) \right].
\end{aligned}
\end{equation}
We denote $\Expect[ L_{n}^{f(1)}L_{n}^{f(2)} | \mathcal{F}_{n} ]
=\psi_n = \psi_n(\Theta^{(1)},\Theta^{(2)},U^{(1)},U^{(2)})$ 
for conciseness. 

Now define 
\begin{equation}
\label{eq:conditioning}
\Omega_{n}^{(1)}:= \left\{
\max_{1\le l_1, l_2 \le \kappa}\left(  \left|\frac{1}{n}\inn{\Theta_{*l_1}^{(1)}}{\Theta_{*l_2}^{(1)}}-\Sigma_{\Theta}(l_1,l_2)\right|, 
\left|\frac{1}{p}\inn{U_{*l_1}^{(1)}}{U_{*l_2}^{(1)}}- \Sigma_{U}(l_1,l_2)\right| \right) \le \delta_{n}
\right\}
\end{equation}
where $\delta_{n} \to 0$ and $\bbP((\Omega_n^{(1)})^c)\to 0$ as $n\to\infty$. 
Such a sequence of $\delta_{n}$ exists due to law of large numbers.
Define $\Omega_{n}^{(2)}$ as an identical and independent copy of $\Omega_{n}^{(1)}$ that depends on ${\Theta}^{(2)}, U^{(2)}$.
Conditioning on $\Theta^{(1)}$, $U^{(1)}$ and $U^{(2)}$, 
% $\{\theta^{(1)}_{l}\}_{1\le l \le \kappa}$ and $\{ \bfu^{(m)}_{l} \}_{1\le l \le \kappa}$ where $m \in \{ 1,2 \}$,
the exponent in (\ref{secmom:exp}) can be written as
\begin{equation}
\begin{split}
% \sum_{l_1=1}^{\kappa} \sum_{l_2=1}^{\kappa}
\sqrt{\frac{n}{p}}\, \inn{Z}{V}
% Z_{l_1, l_2} V_{l_{1},l_{2}}
\end{split}
\end{equation}
where
% $Z_{l_{1},l_{2}}$ and  $V_{l_{1},l_{2}}$ have the specific form as follows.
\begin{equation}\label{def:V}
V = A\,\mathrm{vec}(U^{(2)}) \in \reals^{\kappa^2}\qquad \mbox{for}\quad A = \frac{1}{\sqrt{p}}\, I_\kappa \otimes (U^{(1)})' \in \reals^{\kappa^2\times \kappa p},
\end{equation}
% V=
% \underbrace{\left(
% \begin{array}{llll}
% A_{11}& 0 & \ldots & 0\\
% 0  &  A_{22} & \ldots & 0\\
% \vdots & \vdots & \ddots & \vdots\\
% 0 & 0 & \ldots & A_{\kappa\kappa}
% \end{array}
% \right)_{\kappa^{2} \times \kappa p}}_{A\text{ (say) }}
% \left(
% \begin{array}{l}
% \bfu_{1}^{(2)}\\
% \vdots\\
% \bfu_{\kappa}^{(2)}
% \end{array}
% \right)_{\kappa p \times 1}
% \end{equation}
% where $A_{11}=\ldots= A_{\kappa\kappa}$ are matrices of order $\kappa \times p$ 
% and
% \[
% A_{11}=
% \frac{1}{\sqrt{p}}\left(
% \begin{array}{llll}
% u_{1,1}^{(1)} & \ldots & u_{1,p}^{(1)}\\
% u_{2,1}^{(1)} & \ldots & u_{2,p}^{(1)}\\
% \vdots        & \ddots & \vdots       \\
% u_{\kappa,1}^{(1)} & \ldots & u_{\kappa,p}^{(1)}
% \end{array}
% \right)
% \]
% and
\begin{equation}
	\label{def:Z}
Z = B\,\mathrm{vec}(\Theta^{(2)})\in \reals^{\kappa^2}\qquad
\mbox{for}\quad
B = \frac{1}{\sqrt{n}}\,I_\kappa\otimes (\Theta^{(1)})'\in \reals^{\kappa^2\times \kappa n}.
% Z=
% \underbrace{\left(
% \begin{array}{llll}
% B_{11}& 0 & \ldots & 0\\
% 0  &  B_{22} & \ldots & 0\\
% \vdots & \vdots & \ddots & \vdots\\
% 0 & 0 & \ldots & B_{\kappa\kappa}
% \end{array}
% \right)_{\kappa^{2} \times \kappa n}}_{B\text{ (say) }}
% \left(
% \begin{array}{l}
% \theta_{1}^{(2)}\\
% \vdots\\
% \theta_{\kappa}^{(2)}
% \end{array}
% \right)_{\kappa n \times 1}
\end{equation}
% where $B_{11}=\ldots= B_{\kappa\kappa}$ are matrices of order $\kappa \times n$ and
% \[
% B_{11}=
% \frac{1}{\sqrt{n}}\left(
% \begin{array}{llll}
% \theta_{1,1}^{(1)} & \ldots & \theta_{1,n}^{(1)}\\
% \theta_{2,1}^{(1)} & \ldots & \theta_{2,n}^{(1)}\\
% \vdots        & \ddots & \vdots       \\
% \theta_{\kappa,1}^{(1)} & \ldots & \theta_{\kappa,n}^{(1)}
% \end{array}
% \right)
% \]
Our goal is to prove the random variables $\{\psi_n\, \Indc_{\Omega_n^{(1)}} \Indc_{\Omega_n^{(2)}}\}_{n\geq 1}$ are uniformly integrable. 
To this end, it suffices to show that $\Expect[\psi_n^{(1+\eta)}\Indc_{\Omega_n^{(1)}} \Indc_{\Omega_n^{(2)}}]$ is uniformly bounded for some $\eta>0$.  
Now from assumption on the priors,
% Assumption \ref{ASS:SUB} (i)
we have for sufficiently large values of $n$,
\begin{equation}\label{expectation:psiI}
\begin{split}
\Expect\left[ \psi_n^{(1+\eta)}\, \Indc_{\Omega_n^{(1)}} \Indc_{\Omega_n^{(2)}} \Big| 
\Theta^{(1)}, U^{(1)}, U^{(2)}
% \,\Theta^{(1)},U^{(1)},\left\{\mathbf{u}_{l}^{(2)}\right\}_{1\le l \le \kappa}  \right.
\right]
% \\
& 
= \Expect\left[ \psi_n^{(1+\eta)}\, \Indc_{\Omega_n^{(1)}}\Indc_{\wt{\Omega}_n^{(2)}}\Indc_{\widehat{\Omega}_{n}^{(2)}} \Big|
\Theta^{(1)}, U^{(1)}, U^{(2)}
% \Theta^{(1)},U^{(1)},\left\{\mathbf{u}_{l}^{(2)}\right\}_{1\le l \le \kappa}  \right.
\right]\\
& \le \Expect\left[ \Indc_{\Omega_n^{(1)}} \Indc_{\wt{\Omega}_n^{(2)}}
\exp\Big( \frac{1}{2\gamma} (1+2\eta)^2 V' B D_{\Theta} B' V \Big) \right]. 
\end{split}
\end{equation}
Here 
\begin{align*}
\wt{\Omega}_n^{(2)} & = \left\{ \max_{1\le l_1, l_2 \le \kappa}   
\left| \frac{1}{p} \inn{U_{*l_1}^{(2)}}{U_{*l_2}^{(2)}}-\Sigma_{U}(l_1,l_2)\right|  
\le \delta_{n} \right\},\\
\widehat{\Omega}_{n}^{(2)} & = \left\{ \max_{1\le l_1, l_2 \le \kappa}  
\left| \frac{1}{n}\inn{\Theta_{*l_1}^{(2)}}{\Theta_{*l_2}^{(2)}}-\Sigma_{\Theta}(l_1,l_2)\right| \le \delta_{n}   \right\}
\end{align*}
and 
% $D_{\Theta}$ is a matrix of order $n\kappa \times n\kappa$ of the following form
\begin{equation}\label{form:D}
	D_\Theta = \widetilde{\Sigma}_\Theta \otimes I_n \in \reals^{\kappa n\times \kappa n}.
% D_{\Theta}=
% \left(
% \begin{array}{llll}
% \wt{\Sigma}_{\Theta}(1,1)I_{n \times n}& \wt{\Sigma}_{\Theta}(1,2) I_{n \times n} & \ldots & \wt{\Sigma}_{\Theta}(1,\kappa)I_{n \times n}\\
% \wt{\Sigma}_{\Theta}(2,1)I_{n \times n}& \wt{\Sigma}_{\Theta}(2,2)I_{n \times n}  & \ldots & \wt{\Sigma}_{\Theta}(2,\kappa)I_{n \times n}\\
% \vdots                                 &  \vdots                                  & \ldots & \vdots\\
% \wt{\Sigma}_{\Theta}(\kappa,1)I_{n \times n}& \wt{\Sigma}_{\Theta}(\kappa,2)I_{n \times n}  & \ldots & \wt{\Sigma}_{\Theta}(\kappa,\kappa)I_{n \times n}
% \end{array}
% \right).
\end{equation}
As a consequence, for $B$ defined in \eqref{def:Z}, we have
\begin{equation}
B D_{\Theta} B'= \widetilde{\Sigma}_\Theta \otimes \left[\frac{1}{n}(\Theta^{(1)})'\Theta^{(1)}\right].
% \left(
% \begin{array}{llll}
% \wt{\Sigma}_{\Theta}(1,1)B_{11}B_{11}'& \wt{\Sigma}_{\Theta}(1,2)B_{11}B_{11}'  & \ldots & \wt{\Sigma}_{\Theta}(1,\kappa)B_{11}B_{11}'\\
% \wt{\Sigma}_{\Theta}(2,1)B_{11}B_{11}'& \wt{\Sigma}_{\Theta}(2,2)B_{11}B_{11}'  & \ldots & \wt{\Sigma}_{\Theta}(2,\kappa)B_{11}B_{11}'\\
% \vdots                                 &  \vdots                                  & \ldots & \vdots\\
% \wt{\Sigma}_{\Theta}(\kappa,1)B_{11}B_{11}'& \wt{\Sigma}_{\Theta}(\kappa,2)B_{11}B_{11}'  & \ldots & \wt{\Sigma}_{\Theta}(\kappa,\kappa)B_{11}B_{11}'
% \end{array}
% \right).
\end{equation}
Recall that for any matrix $A$, let $\maxnorm{A} = \max_{i,j}|A_{ij}|$ be the vector $\ell_\infty$-norm of $A$.
On the event $\wt{\Omega}_{n}^{(2)} \cap  \Omega_n^{(1)}$, we have $\maxnorm{BD_{\Theta}B'- \wt{\Sigma}_{\Theta}\otimes {\Sigma}_{\Theta}} < \maxnorm{\wt{\Sigma}_{\Theta}}\delta_{n}$. Now we know that for any symmetric matrix $\Sigma$ of dimension $\kappa^{2} \times \kappa^{2}$, $\|\Sigma\|_{2} \le \| \Sigma \|_{\mathrm{F}} \le \kappa^{2} \maxnorm{\Sigma}$ where $\|\cdot \|_{2}$ and $\| \cdot \|_{\mathrm{F}}$ denote the spectral norm and  Frobenius norm respectively. 
So 
\begin{align*}
\Indc_{\Omega_n^{(1)}} \Indc_{\wt{\Omega}_n^{(2)}}V' B D_{\Theta} B' V  
&\leq \Indc_{\Omega_n^{(1)}} \Indc_{\wt{\Omega}_n^{(2)}} V' \Big(\wt{\Sigma}_{\Theta} \otimes \Sigma_{\Theta}+ \kappa^2 \maxnorm{\wt{\Sigma}_{\Theta}}\delta_{n}I_{\kappa^2} \Big)V\\
& \leq \Indc_{\Omega_n^{(1)}} \Indc_{\wt{\Omega}_n^{(2)}} V' \Big(\wt{\Sigma}_{\Theta} \otimes \Sigma_{\Theta}+ \delta'_{n}I_{\kappa^2} \Big)V
\end{align*}
where $\delta'_n\to 0$ is a sequence depending only on $\kappa$, $\widetilde{\Sigma}_\Theta$ and $\delta_n$.
Therefore, we have
\begin{equation}
\begin{split}
&~~~~\Expect\left[ \Indc_{\Omega_n^{(1)}} \Indc_{\wt{\Omega}_n^{(2)}}
 \exp\Big( \frac{1}{2\gamma} (1+2\eta)^2 V' B D_{\Theta} B' V \Big) 
 \right]\\
 & \leq \Expect\left[ \Indc_{\Omega_n^{(1)}} \Indc_{\wt{\Omega}_n^{(2)}}\exp\Big( \frac{1}{2\gamma}  (1+2\eta)^2
 V' \big(\wt{\Sigma}_{\Theta} \otimes \Sigma_{\Theta}+ \delta'_{n}I_{\kappa^2} \big)V \Big) \right]\\
 &=  \Expect\left[ \Indc_{\Omega_n^{(1)}} \Expect\left[\Indc_{\wt{\Omega}_n^{(2)}}\exp\Big( \frac{1}{2\gamma} (1+2\eta)^2 V' \big(\wt{\Sigma}_{\Theta} \otimes \Sigma_{\Theta}+ \delta'_{n}I_{\kappa^2} \big)V \Big) \Big| \Theta^{(1)}, U^{(1)} 
 \right]\right]\\
 & \le \Expect\left[ \Indc_{\Omega_n^{(1)}} \Expect\left[
 \exp\Big( \frac{1}{2\gamma} (1+2\eta)^2 V' \big(\wt{\Sigma}_{\Theta} \otimes \Sigma_{\Theta}+ \delta'_{n}I_{\kappa^2} \big)V \Big) \Big| \Theta^{(1)}, U^{(1)}  \right]\right]\\
 &= \Expect \left[ \Indc_{\Omega_n^{(1)}} \exp\Big( \frac{1}{2\gamma} (1+2\eta)^2 V' \big(\wt{\Sigma}_{\Theta} \otimes \Sigma_{\Theta}+ \delta'_{n}I_{\kappa^2} \big)
 V \Big)\right].
\end{split}
\end{equation} 
In {Step 3} we prove that the sequence 
\begin{equation}
	\label{eq:UI}
\limsup_{n\to\infty}\Expect \left[ \Indc_{\Omega_n^{(1)}} \exp\Big( \frac{1}{2\gamma} (1+2\eta)^2 V' \big(\wt{\Sigma}_{\Theta} \otimes \Sigma_{\Theta}+ \delta'_{n}I_{\kappa^2} \big)
 V \Big)\right] < \infty\,\,\, \mbox{for some $\eta>0$.}
\end{equation}
If we assume \eqref{eq:UI}, the rest of the proof can be completed as follows. Observe that by central limit theorem 
\begin{equation*}
\psi_n\, \Indc_{\Omega_n^{(1)}} \Indc_{\Omega_n^{(2)}} \stackrel{d}{\to}
\exp\left( \frac{1}{\sqrt{\gamma}} \sum_{l_1=1}^{\kappa} \sum_{l_2=1}^{\kappa} T_{l_1,l_2} Y_{l_1,l_2} \right)
\end{equation*}
% $\psi_n\, \Indc_{\Omega_n^{(1)}} \Indc_{\Omega_n^{(2)}}$ converges in distribution to $\exp\left( \frac{1}{\sqrt{\gamma}} \sum_{l=1}^{\kappa} \sum_{l'=1}^{\kappa} M_{l,l'} N_{l,l'} \right)$
where 
$\frac{1}{\sqrt{n}}\inn{\theta_{l_1}^{(1)}}{\theta_{l_2}^{(2)}} \stackrel{d}{\to} T_{l_1,l_2}$ and $\frac{1}{\sqrt{p}}\inn{\bfu_{l_{1}}^{(1)}}{\bfu_{l_{2}}^{(2)}}\stackrel{d}{\to}Y_{l_1,l_2}$.
In addition, the collections $\{T_{l_1,l_2}\}$ and $\{Y_{l_1,l_2} \}$ are mutually independent. 
Furthermore, the random variables $T_{l_1,l_2}$ are jointly Gaussian with mean $0$ and $\Cov(T_{l_{1},l_{2}},T_{l_{3},l_{4}})= \Sigma_{\Theta}(l_{1},l_{3})\Sigma_{\Theta}(l_{2},l_{4})$ and 
analogous results hold for $\{Y_{l_1,l_2}\}$.
% $\Cov(N_{l_{1},l_{2}},N_{l_{3},l_{4}})= \Sigma_{U}(l_{1},l_{3})\Sigma_{U}(l_{2},l_{4}).$
Let $T = (T_{l_1,l_2})$ and $Y = (Y_{l_1,l_2})$ be $\kappa\times \kappa$ matrices.
Then the foregoing discussion implies that $\mathrm{vec}(T)\sim N_{\kappa^2}(0, \Sigma_\Theta\otimes \Sigma_\Theta)$ and is independent of 
$\mathrm{vec}(Y) \sim N_{\kappa^2}(0, \Sigma_U\otimes \Sigma_U)$.
This, together with the uniform integrability of $\psi_n\Indc_{\Omega_n^{(1)}} \Indc_{\Omega_n^{(2)}}$, implies that
% We call the covariance matrix of $\{ M_{l,l'} \}$ $(\Sigma_\Theta\otimes \Sigma_\Theta)$ and the covariance matrix  of $\{ N_{l,l'} \}$ $(\Sigma_U\otimes \Sigma_U)$. From uniform integrability of $\psi_n\Indc_{\Omega_n^{(1)}} \Indc_{\Omega_n^{(2)}}$, we have
\begin{equation}
\begin{split}
\lim_{n\to\infty}\Expect\left[ \Indc_{\Omega_n^{(1)}} \Indc_{\Omega_n^{(2)}} \psi_n\right] &= \Expect\left[ \exp\left( \frac{1}{\sqrt{\gamma}}
\inn{\mathrm{vec}(T)}{\mathrm{vec}(Y)} 
% \sum_{l=1}^{\kappa} \sum_{l'=1}^{\kappa} M_{l,l'} N_{l,l'}
\right) 
\right]
\\
&
= \Expect\left[ \exp\left( \frac{1}{2\gamma} \mathrm{vec}(Y)' (\Sigma_{\Theta}\otimes \Sigma_\Theta) \mathrm{vec}(Y) \right) \right]\\
&= \exp\left\{ \frac{1}{2} \sum_{i=1}^{\kappa^{2}} \log\left(1- \frac{\lambda_{i}}{\gamma}\right) \right\}\\
&= \exp\left\{ \sum_{k=1}^{\infty} \frac{\sum_{i=1}^{\kappa^2} \lambda_{i}^{k}}{2k\gamma^{k}} \right\}
\\
% &= \exp\left( \sum_{k=1}^{\infty} \frac{\Tr\left((\Sigma_U\otimes \Sigma_U)^{\frac{1}{2}}(\Sigma_\Theta\otimes \Sigma_\Theta) (\Sigma_U\otimes \Sigma_U)^{\frac{1}{2}} \right)^{k}}{2k\gamma^{k}} \right)\\
&
= \exp\left\{ \sum_{k=1}^{\infty} \frac{\Tr\left( (\Sigma_\Theta\Sigma_U)^k \otimes (\Sigma_\Theta\Sigma_U)^k \right)}{2k\gamma^{k}}  \right\}.
\end{split}
\end{equation}
Here 
% $N:=\left(N_{l,l'}\right)_{1\le l,l'\le \kappa}$ and
$\{\lambda_{i}\}_{1\le i \le \kappa^2}$ are the eigenvalues of the matrix $(\Sigma_U\otimes \Sigma_U)^{{1}/{2}}(\Sigma_\Theta\otimes \Sigma_\Theta) (\Sigma_U\otimes \Sigma_U)^{{1}/{2}}$.
We complete the proof by noting that $\Tr( (\Sigma_\Theta\Sigma_U)^k \otimes (\Sigma_\Theta\Sigma_U)^k)= [\Tr((\Sigma_\Theta \Sigma_U)^k)]^2 =  \mu_{k}^{2}$. 
% Now
% \begin{equation}
% \begin{split}
% &\Tr( (\Sigma_\Theta\Sigma_U)^k \otimes (\Sigma_\Theta\Sigma_U)^k)
% = \Tr( (\Sigma_\Theta\otimes \Sigma_\Theta)  (\Sigma_U \otimes\Sigma_U))^{k} \\
% & = \sum_{i_{0},i_{1},\ldots, i_{k-1}} (\Sigma_\Theta\otimes \Sigma_\Theta)(i_{0},i_{1}) (\Sigma_U\otimes \Sigma_U)(i_{1},i_{2})\ldots (\Sigma_U\otimes \Sigma_U)(i_{2k-1},i_{0})\\
% &= \sum_{l_{0}^{(1)},\ldots ,l_{2k-1}^{(1)}, l_{0}^{(2)},\ldots, l_{2k-1}^{(2)}} \Sigma_{\Theta}(l_{0}^{(1)},l_{1}^{(1)}) \Sigma_{\Theta}(l_{0}^{(2)},l_{1}^{(2)})\Sigma_{U}(l_{1}^{(1)},l_{2}^{(1)})\Sigma_{U}(l_{1}^{(2)},l_{2}^{(2)}) \ldots \Sigma_{U}(l_{2k-1}^{(1)},l_{0}^{(1)})\Sigma_{U}(l_{2k-1}^{(2)},l_{0}^{(2)}).\\
% &= \left(\sum_{l_{0}^{(1)},\ldots ,l_{2k-1}^{(1)}} \Sigma_{\Theta}(l_{0}^{(1)},l_{1}^{(1)})\Sigma_{U}(l_{1}^{(1)},l_{2}^{(1)}) \ldots  \Sigma_{U}(l_{2k-1}^{(1)},l_{0}^{(1)}) \right)\times \\
% &~~~~~~~~~~~~~~~~~~~~~~~~~~~\left( \sum_{l_{0}^{(2)},\ldots ,l_{2k-1}^{(2)}} \Sigma_{\Theta}(l_{0}^{(2)},l_{1}^{(2)})\Sigma_{U}(l_{1}^{(2)},l_{2}^{(2)}) \ldots  \Sigma_{U}(l_{2k-1}^{(2)},l_{0}^{(2)})\right)\\
% &= \mu_{k}^{2}
% \end{split}
% \end{equation}
% as required.

\paragraph{Step 3.}
In the final step of the proof, we verify \eqref{eq:UI}. Recall (\ref{def:V}) to observe that 
\begin{equation}
\begin{split}
 &\exp\left( \frac{1}{2\gamma}(1+2\eta)^2  V' \big(\wt{\Sigma}_{\Theta} \otimes \Sigma_{\Theta}+ \delta'_{n}I_{\kappa^2} \big)V \right)\\
 &= \exp\left( \frac{1}{2\gamma}(1+2\eta)^2 \mathrm{vec}(U^{(2)})' A'\big(\wt{\Sigma}_{\Theta} \otimes \Sigma_{\Theta}+ \delta'_{n}I_{\kappa^2} \big)A \mathrm{vec}(U^{(2)}) \right).
\end{split}
\end{equation}
% where
%  \[\bfu^{(2)}= \left(
% \begin{array}{l}
% \bfu_{1}^{(2)}\\
% \vdots\\
% \bfu_{\kappa}^{(2)}
% \end{array}
% \right)_{\kappa p \times 1}.\]
Now write 
\[
\wt{U}^{(2)}= D_{U}^{-{1}/{2}} U^{(2)}
\]
where 
\begin{equation}\label{form:D}
D_U = \wt{\Sigma}_U \otimes I_p \in \reals^{\kappa p\times \kappa p}.
% D_{U}=
% \left(
% \begin{array}{llll}
% \wt{\Sigma}_{U}(1,1)I_{p \times p}& \wt{\Sigma}_{U}(1,2) I_{p \times p} & \ldots & \wt{\Sigma}_{U}(1,\kappa)I_{p \times p}\\
% \wt{\Sigma}_{U}(2,1)I_{p \times p}& \wt{\Sigma}_{U}(2,2)I_{p \times p}  & \ldots & \wt{\Sigma}_{U}(2,\kappa)I_{p \times p}\\
% \vdots                                 &  \vdots                                  & \ldots & \vdots\\
% \wt{\Sigma}_{U}(\kappa,1)I_{p \times p}& \wt{\Sigma}_{U}(\kappa,2)I_{p \times p}  & \ldots & \wt{\Sigma}_{U}(\kappa,\kappa)I_{p \times p}
% \end{array}
% \right).
\end{equation}
So we have 
\begin{equation}
\begin{split}
 & \exp\left( \frac{1}{2\gamma}(1+2\eta)^2 (\mathrm{vec}(U)^{(2)})' A'\left(\wt{\Sigma}_{\Theta} \otimes \Sigma_{\Theta}+ \delta'_{n}I_{\kappa^2} \right)A \mathrm{vec}(U)^{(2)} \right)\\
 &=  \exp\left( \frac{1}{2\gamma}(1+\eta) \mathrm{vec}(\wt{U}^{(2)})' D_{U}^{{1}/{2}}A'\left(\wt{\Sigma}_{\Theta} \otimes \Sigma_{\Theta}+ \delta'_{n}I_{\kappa^2} \right) A D_{U}^{{1}/{2}} \mathrm{vec}(\wt{U}^{(2)}) \right).
\end{split}
\end{equation}
Theorem 1 from \cite{hsu12} implies for any non-random non-negative definite $\widetilde{\Sigma}$ and all $t > 0$,
\begin{equation}\label{HSU}
\Prob\left[ \left( \mathrm{vec}(\wt{U}^{(2)})' \widetilde{\Sigma}\mathrm{vec}(\wt{U}^{(2)}) \right)> \Tr(\wt{\Sigma})+ \sqrt{\Tr(\wt{\Sigma}^2)t}+ 2 \|\wt{\Sigma}\|_2 t \right]\le e^{-t}.
\end{equation}
In particular, the tail bound in (\ref{HSU}) only depends on the nonzero eigenvalues of $\wt{\Sigma}$.
Now the nonzero eigenvalues of 
$$
D_{U}^{{1}/{2}}A'\left(\wt{\Sigma}_{\Theta} \otimes \Sigma_{\Theta}+ \delta'_{n}I_{\kappa^2} \right) AD_{U}^{{1}/{2}}
$$ 
are the same as those of 
$$ 
A D_{U} A' \left(\wt{\Sigma}_{\Theta} \otimes \Sigma_{\Theta}+ \delta'_{n}I_{\kappa^2} \right).
$$ 
However on $\Omega_n^{(1)}$, we have $AD_{U}A'= \wt{\Sigma}_{U} \otimes \Sigma_{U} + P  $ where $P$ is a perturbation matrix with $\maxnorm{P}= O(\delta_n)$. 
As a consequence, Theorem 5.5.4 of \cite{stsu90} implies that the nonzero eigenvalues of
$$ A D_{U} A' \left(\wt{\Sigma}_{\Theta} \otimes \Sigma_{\Theta}+ \delta'_{n}I_{\kappa^2} \right)$$ are the eigenvalues of $$ (\wt{\Sigma}_{U} \otimes \Sigma_{U})(\wt{\Sigma}_{\Theta} \otimes \Sigma_{\Theta})+ O(\delta_{n}) .$$ 
Here the constant in the $O(\delta_{n})$ term depends on the eigenvalues of $(\wt{\Sigma}_{U} \otimes \Sigma_{U})(\wt{\Sigma}_{\Theta} \otimes \Sigma_{\Theta})$ and $\gamma$, but not on $n$ and $p$. 

For convenience, we define 
$\wt{\Sigma}:= D_{U}^{{1}/{2}}A' (\wt{\Sigma}_{\Theta} \otimes \Sigma_{\Theta}+ \delta'_{n}I_{\kappa^2} ) AD_{U}^{{1}/{2}}$.
On $\Omega_{n}^{(1)}$, $\Tr(\wt{\Sigma})$ and $\Tr(\wt{\Sigma}^2)$ are uniformly bounded. 
So given any $\epsilon>0$, there exists a sufficiently large $t_0 > 0$ that is independent of $n$ such that for all $t \geq t_0$,
\[
\frac{\Tr(\wt{\Sigma})+ \sqrt{\Tr(\wt{\Sigma}^2)t}}{t} \le \epsilon.
 % ~~ \forall ~ t\ge t_{0}.
\] 
So for all $t>t_{0}$ we have
\begin{equation*}
\begin{split}
&\Indc_{\Omega_n^{(1)}}\Prob\left[ \frac{1}{\gamma}\mathrm{vec}(\wt{U}^{(2)})' \wt{\Sigma}\mathrm{vec}(\wt{U}^{(2)}) > \left(2 \frac{\|\wt{\Sigma}\|_{2}}{\gamma}+\epsilon\right)t \left |   \Theta^{(1)},U^{(1)}\right.\right] \le e^{-t},
% \\
% & \Rightarrow \Indc_{\Omega_n^{(1)}}\Prob\left[ \frac{1}{\gamma}\mathrm{vec}(\wt{U}^{(2)})' \wt{\Sigma}\mathrm{vec}(\wt{U}^{(2)}) > \left(2 \frac{\|(\wt{\Sigma}_{U} \otimes \Sigma_{U})(\wt{\Sigma}_{\Theta} \otimes \Sigma_{\Theta})\|_{2}}{\gamma}+O(\delta_n)+\epsilon\right)t \left |   \Theta^{(1)},U^{(1)}\right.\right] \le e^{-t}
\end{split}
\end{equation*}
% uniformly on $\Omega_{n}^{(1)}$. Now
and hence
\begin{equation}\label{final}
\begin{split}
&\Indc_{\Omega_n^{(1)}} \Prob\left[ \exp\left( \frac{1}{2\gamma} (1+ 2\eta)^2 \mathrm{vec}(\wt{U}^{(2)})' \wt{\Sigma}\mathrm{vec}(\wt{U}^{(2)}) \right)>t \left |   \Theta^{(1)},U^{(1)}\right.\right]\\
&\le \left(\frac{1}{t}\right)^{\frac{2}{(1+2\eta)^2
(2{\|\wt{\Sigma}\|_{2}}/{\gamma}+\epsilon)
% \left(2 \frac{\|(\wt{\Sigma}_{U} \otimes \Sigma_{U})(\wt{\Sigma}_{\Theta} \otimes \Sigma_{\Theta})\|_{2}}{\gamma}+O(\delta_n)+\epsilon\right)
}
}
\end{split}
\end{equation}
Since ${\|(\wt{\Sigma}_{U} \otimes \Sigma_{U})(\wt{\Sigma}_{\Theta} \otimes \Sigma_{\Theta})\|_{2}} < {\gamma}$, we can choose $\epsilon$ and $\eta$ small enough such that on $\Omega_n^{(1)}$,
\[
\frac{2}{(1+2\eta)^2
(2{\|\wt{\Sigma}\|_{2}}/{\gamma}+\epsilon)
}
% \frac{2}{(1+\eta)\left(2 \frac{\|(\wt{\Sigma}_{U} \otimes \Sigma_{U})(\wt{\Sigma}_{\Theta} \otimes \Sigma_{\Theta})\|_{2}}{\gamma}+O(\delta_n)+\epsilon\right)}
\ge  \alpha_{0}>1.
\]
Hence we have the last expression in (\ref{final}) is bounded from above by ${t^{-\alpha_0}}$. As a consequence, 
\[
\Expect\left[\Indc_{\Omega_{n}^{(1)}}\exp\left(  \frac{1}{2\gamma} (1+ 2\eta)^2 \mathrm{vec}(\wt{U}^{(2)})' \wt{\Sigma}\mathrm{vec}(\wt{U}^{(2)})  \right)  \right]
\]
is uniformly bounded. This completes the proof.

\subsection{Proof of part 3}
We have from the proof of Proposition \ref{prop:norcont} that for any given $\epsilon,\delta >0$ there exists $K=K(\epsilon,\delta)$ and for any subsequence $n_l$ there exists a further subsequence $n_{l_q}$ such that 
\begin{equation}\label{bdd:liketoseriesIII}
\mathbb{P}_{n_{l_q}}\left[ \left| \log(L_{n_{l_q}}) - 
\sum_{k=1}^{K} \frac{2\mu_{k}(B_{n_{l_q},k}-p\Indc_{k=1})-\mu_{k}^{2}}{2\sigma_{k}^{2}} \right|\ge \frac{\epsilon}{2} \right] \le  \frac{\delta}{2}.
\end{equation} 
Now choose $K'\ge K$ such that 
\[
\sum_{K'+1}^{\infty}\frac{\mu_{k}^{2}}{\sigma_{k}^{2}}\le \max\left\{\frac{\delta\epsilon^{2}}{100},\frac{\epsilon}{100}\right\}.
\]
Now observe that for any $k_{1}<k_{2}<m_{n}= o(\sqrt{\log n})$, $\Expect_{\mathbb{P}_{n}}\left[ B_{n,k_{1}} \right]=0$, $\Cov(B_{n,k_{1}},B_{n,k_{2}})=0$ and $\Var(B_{n,k_{i}})= 2k_{1}\gamma^{k_{i}}\left(1+ O(\frac{k_{i}^{2}}{n})\right)$ for $i \in \{1,2 \}.$
So 
\[
\Var\left[\sum_{k=K'+1}^{m_{n_{l_q}}} \frac{2\mu_{k} B_{n_{l_q},k}-\mu_{k}^{2}}{2\sigma_{k}^{2}}\right]= \left(1+ O\left(\frac{m_{n_{l_q}}^{2}}{n_{l_q}}\right)\right)\sum_{k=K'+1}^{m_{n_{l_q}}}\frac{\mu_{k}^{2}}{\sigma_{k}^{2}}\le \frac{\delta\epsilon^{2}}{100}.
\]
Now for large value of $n_{l_q}$,
\begin{equation}\label{bdd:series}
\begin{split}
&\mathbb{P}_{n_{l_q}}\left[\left|\sum_{k=K+1}^{m_{n_{l_q}}} \frac{2\mu_{k}B_{n_{l_q},k}}{\sigma_{k}^{2}}\right| \ge \frac{\epsilon}{4}\right] \le \frac{16 \delta\epsilon^{2}}{100\epsilon^{2}},
\qquad \mbox{and so} \\
& \mathbb{P}_{n_{l_q}} \left[\left|\sum_{k=K+1}^{m_{n_{l_q}}} \frac{2\mu_{k}\left(B_{n_{l_q},k}\right)-\mu_{k}^{2}}{2\sigma_{k}^{2}}  \right| \ge \frac{\epsilon}{4}+ \frac{\epsilon}{100}  \right]\le\frac{16 \delta\epsilon^{2}}{100\epsilon^{2}}.
\end{split}
\end{equation}

\noindent Plugging in the estimates of (\ref{bdd:liketoseriesIII}) and (\ref{bdd:series}) we have for all large values of $n_{l_q}$,
\begin{equation}
	\label{eq:subseq-conv}
% \begin{split}
\mathbb{P}_{n_{l_q}}\left[  \left| \log(L_{n_{l_q}}) - \sum_{k=1}^{m_{n_{l_q}}} \frac{2\mu_{k}(B_{n_{l_q},k}- p\Indc_{k=1})-\mu_{k}^{2}}{2\sigma_{k}^{2}}
\right| \ge \epsilon \right]\le \delta.
% \\
% \Rightarrow\mathbb{P}_{n}\left[  \left| \log(L_{n}) - \left\{\sum_{k=1}^{m_{n}} \frac{2\mu_{k}\left(B_{n,k}- p\Indc_{k=1}\right)-\mu_{k}^{2}}{2\sigma_{k}^{2}}\right\}\right|\ge \epsilon \right] \to 0.
% \end{split}
\end{equation}
Since \eqref{eq:subseq-conv} occurs to any subsequence and any $(\epsilon,\delta)$ pair, this completes the proof.  
\hfill{$\square$}

%% file: contiguity.tex
% !TEX root = GAU_MUL.tex

\section{Appendix: Proof of Proposition \ref{prop:norcont}}
At first we introduce the concept of Wasserstein's metric which will be used in the proof of Proposition \ref{prop:norcont}. 
% \begin{definition}\label{def:wass}
Let $F$ and $G$ be two distribution functions with finite $p$-th moment. Then the Wasserstein distance $W_p$ between $F$ and $G$ is defined to be 
\[
W_p(F,G)= \left[  \inf_{X \sim F, Y \sim G} \Expect|X-Y|^{p} \right]^{{1}/{p}}.
\]
Here $X$ and $Y$ are random variables having distribution functions $F$ and $G$ respectively.   
% \end{definition}
The following result will be useful in our proof. 
See, for instance, \citet{Mal72} for its proof.
\begin{prop}\label{prop:wasser}
Suppose $F_n$ be a sequence of distribution functions and $F$ be a distribution function. Then $F_n\stackrel{d}{\to} F$ in distribution and $\int x^2 dF_n(x)\to \int x^2dF(x)$ if $W_2(F_n,F)\to 0$.  
\end{prop}
% The proof of Proposition \ref{prop:wasser} is well known. One might look at Mallows(1972)\cite{Mal72} for a reference.

%With Proposition \ref{prop:useI} in hand, we now state the most important result in this section. This result will be used to prove Theorem \ref{thm:main}. 
% \noindent
% We are now ready to give a proof of Proposition \ref{prop:norcont}

\begin{proof}[Proof of Proposition \ref{prop:norcont}]
We now prove the proposition.
\paragraph{Proof of mutual contiguity and \eqref{eq:lr-limit}:}
This proof is broken into two steps.
We focus on proving \eqref{eq:lr-limit}.
Given \eqref{eq:lr-limit}, mutual contiguity is a direct consequence of Le Cam's first lemma \cite{LeCam}.

\textbf{Step 1.}
We first prove the random variable on the righthand side of (\ref{eq:lr-limit}) is almost surely positive and has mean $1$. 
Let us define 
\[
L:=\exp\left\{ \sum_{i=1}^{\infty}\frac{2\mu_{i}Z_{i}-\mu_{i}^{2}}{2\sigma_{i}^{2}} \right\},\qquad
% \]
% and
% \[
L^{(m)}:= \exp\left\{ \sum_{i=1}^{m}\frac{2\mu_{i}Z_{i}-\mu_{i}^{2}}{2\sigma_{i}^{2}} \right\},\quad
\forall m\in \naturals.
\]
As $Z_i \sim N(0,\sigma_{i}^{2})$, for any $i\in\naturals$, and so
\[
\Expect\left[\frac{2\mu_{i}Z_{i}-\mu_{i}^{2}}{2\sigma_{i}^{2}} \right]=1.
\]
So $\{L^{(m)}\}_{m=1}^{\infty}$ is a martingale sequence and 
\[
\Expect\left[ \big(L^{(m)}\big)^2 \right]=\prod_{i=1}^{m} \exp\left\{ \frac{\mu_{i}^{2}}{\sigma_{i}^{2}} \right\}=\exp\left\{ \sum_{i=1}^{m} \frac{\mu_{i}^{2}}{\sigma_{i}^{2}} \right\}.
\]
Now by the righthand side of (\ref{eq:lr-square}),
$L^{(m)}$ is a $L^2$ bounded martingale.
Hence, $L$ is a well defined random variable with
\[
\Expect[L] = 1,\qquad
\Expect[L^2]= \exp\left\{ \sum_{i=1}^{\infty} \frac{\mu_{i}^{2}}{\sigma_{i}^{2}} \right\}.
\]
On the other hand $\log(L)$ is a limit of Gaussian random variables, hence $\log(L)$ is Gaussian with
\[
\Expect[\log(L)]= -\frac{1}{2} \sum_{i=1}^{\infty} \frac{\mu_{i}^{2}}{\sigma_{i}^{2}}, \qquad \Var[\log(L) ]= \sum_{i=1}^{\infty} \frac{\mu_{i}^{2}}{\sigma_{i}^{2}}.
\]
Hence $\mathbb{P}[L=0]= \mathbb{P}[\log(L)=-\infty]=0$. 

\textbf{Step 2.}
Now we prove $Y_n \stackrel{d}{\to} L$.
Since
\[
\limsup_{n \to \infty}\Expect_{\mathbb{P}_n}\left[ Y_n^2\right]<\infty,
\]
condition (iv) implies that the sequence $Y_n$ is tight. 
Prokhorov's theorem further implies that there is a subsequence 
$\{ n_{k} \}_{k=1}^{\infty}$ such that $Y_{n_k}$ converge in distribution to some random variable $L(\{ n_{k} \})$. 
In what follows, we prove that the distribution of $L(\{ n_{k} \})$ does not depend on the subsequence $\{ n_{k} \}$. In particular, $L(\{ n_{k} \})\stackrel{d}{=} L$.
To start with, note that since $Y_{n_k}$ converges in distribution to $L(\{ n_{k} \})$, for any further subsequence $\{ n_{k_l} \}$ of $\{ n_{k}\}$, $Y_{n_{k_l}}$ also converges in distribution to $L(\{ n_{k} \})$.

Given any fixed $\epsilon>0$ take $m$ large enough such that 
\[
\exp\left\{ \sum_{i=1}^{\infty} \frac{\mu_{i}^{2}}{\sigma_{i}^{2}} \right\}-\exp\left\{ \sum_{i=1}^{m} \frac{\mu_{i}^{2}}{\sigma_{i}^{2}} \right\} < \epsilon.
\]
For this fixed number $m$, consider the joint distribution of 
$(Y_{n_{k}},W_{n_k,1},\ldots,W_{n_k,m})$. 
This sequence of $m+1$ dimensional random vectors with respect to $\mathbb{P}_{n_k}$ is tight by condition (ii). 
So it has a further subsequence such that $$(Y_{n_{k_l}},W_{n_{k_l},1},\ldots,W_{n_{k_l},m})|\mathbb{P}_{n_{k_l}}\stackrel{d}{\to}\left(L(\{ n_{k} \}),Z_{3}\ldots,Z_{m}\right).$$ 
% Here we have used condition $ii)$ for the convergence of $(W_{n_{k_l},1},\ldots,W_{n_{k_l},m})|\mathbb{P}_{n_{k_l}}$.
We are to show that we can define the random variables $L^{(m)}$ and $L(\{ n_{k} \})$ in such a way that there exist suitable $\sigma$-algebras $\mathcal{F}_1 \subset\mathcal{F}_2$ such that 
$L^{(m)} \in \mathcal{F}_1$, 
$L(\{ n_{k} \}) \in \mathcal{F}_2$,
and $\Expect \left[ L(\{ n_{k} \})\left|\right. \mathcal{F}_{1} \right]= L^{(m)}$.

% From condition $iv)$ we have
Since $\limsup_{n \to \infty}\Expect_{\mathbb{P}_n}\left[ Y_n^2\right]< \infty$, the sequence $Y_{n_{k_l}}$ is uniformly integrable. 
This, together with condition (i), leads to
\begin{equation}\label{eqn_expder}
\Expect[L(\{ n_{k} \})] = \lim_{l\to\infty} 
\Expect_{\mathbb{P}_{n_{k_l}}} [ Y_{n_{k_l}} ] = 1.
% 1=\Expect_{\mathbb{P}_{n_{k_l}}}\left[ Y_{n_{k_l}} \right] \to \Expect[L(\{ n_{k} \})]=1.
\end{equation}
Now take any positive bounded continuous function $f:\mathbb{R}^m \to \mathbb{R}$. By Fatou's lemma 
\begin{equation}
	\label{eqn_ineq}
\liminf_{l\to\infty} 
\Expect_{\mathbb{P}_{n_{k_l}}} \left[f (W_{n_{k_l},1},\ldots,X_{n_{k_l},m} )Y_{n_{k_l}} \right] \ge \Expect \left[ f\left(Z_1,\ldots,Z_{m}\right)L(\{ n_{k} \})\right].
\end{equation}
However for any constant $\xi$, \eqref{eqn_expder} implies
$\xi=\xi\Expect_{\mathbb{P}_{n_{k_l}}}[ Y_{n_{k_l}} ] \to \xi\Expect[L(\{ n_{k} \})]= \xi$.
%  % we have
% \[
% \xi=\xi\Expect_{\mathbb{P}_{n_{k_l}}}\left[ Y_{n_{k_l}} \right] \to \xi\Expect[L(\{ n_{k} \})]= \xi
% \]
% from (\ref{eqn_expder}).
So (\ref{eqn_ineq}) holds for any bounded continuous function $f$. 
On the other hand, replacing $f$ by $-f$ we have 
\begin{equation}
	\label{eqn_ineqII}
\lim_{l\to\infty} 
\Expect_{\mathbb{P}_{n_{k_l}}}
\left[f(W_{n_{k_l},1},\ldots,W_{n_{k_l},m})Y_{n_{k_l}} \right] 
= \Expect \left[ f(Z_1,\ldots,Z_{m})L(\{ n_{k} \})\right].
\end{equation}
Now condition (ii) leads to
\begin{equation*}
\int f(W_{n_{k_l},1},\ldots,W_{n_{k_l},m})Y_{n_{k_l}} \mathrm{d}\mathbb{P}_{n_{k_l}}= \int f(W_{n_{k_l},1},\ldots,W_{n_{k_l},m})\mathrm{d}\mathbb{Q}_{n_{k_l}} \to \int f(Z_1',\ldots,Z_{m}') \mathrm{d}Q.
\end{equation*}
Here  $Q$ is the measure induced by $(Z_1',\ldots,Z_{m}')$. In particular, one can take the measure $Q$ such that
% it is defined on $(\Omega(\{ n_{k_l} \}),\mathcal{F}(\{ n_{k_l} \}))$ 
$(Z_1,\ldots, Z_{m})$ themselves are distributed as $(Z_1',\ldots,Z_{m}')$ under the measure $Q$. 
This is true since 
\[
\int f(Z_1',\ldots, Z_{m}') \mathrm{d}Q= \Expect\left[ f(Z_1,\ldots, Z_{m}) L^{(m)}\right].
\]
for any bounded continuous function $f$, and so
%Here 
%\[
%V^{(m)}:= \exp\left\{ \sum_{i=2}^{m-1}\frac{2t^{\frac{i+1}{2}}H_{i}-t^{i+1}}{4(i+1)} \right\} \stackrel{d}{=} W^{(m)}.
%\]
% Since $f$ is any bounded continuous function, we have
% \[
$\int_{A}  \mathrm{d}Q= \Expect[ \mathbf{1}_{A} L^{(m)} ]$
% \]
for any $A \in \sigma(Z_1,\ldots, Z_{m})$. 
% Here for any set $A$, $\mathbf{1}_{A}$ denotes the indicator function taking value one on $A$.
Now looking back into (\ref{eqn_ineqII}), we have for any $A \in \sigma(Z_1,\ldots, Z_{m})$,
$\Expect[ \mathbf{1}_{A} L^{(m)} ]= \Expect\left[ \mathbf{1}_{A} L(\{ n_{k} \}) \right]$.
Since $L^{(m)}$ is $\sigma(Z_1,\ldots, Z_{m})$ measurable, we have 
\begin{equation*}
L^{(m)}=\Expect\left[ L(\{ n_{k} \}) \left|\right. \sigma(Z_1,\ldots, Z_{m})  \right].	
\end{equation*}

From Fatou's lemma 
\[
\Expect[ L(\{ n_{k} \})^2]\le \liminf_{n \to \infty} \Expect_{\mathbb{P}_n}[Y_n^2]= \exp\left\{ \sum_{i=1}^{\infty} \frac{\mu_i^{2}}{\sigma_{i}^{2}} \right\}.
\]
As a consequence, we have 
\[
0 \le \Expect|L(\{ n_{k} \})-L^{(m)}|^2 = \Expect[L(\{ n_{k} \})^2]-\Expect[L^{(m)2}]< \epsilon.
\]
So $L_2(F^{L^{(m)}},F^{L(\{ n_{k} \})})< \sqrt{\epsilon}$. Here $F^{L^{(m)}}$ and $F^{L(\{ n_{k} \})}$ denote the distribution functions corresponding to $L^{(m)}$ and $L(\{ n_{k} \})$ respectively. As a consequence, $W_2(F^{L^{(m)}},F^{L(\{ n_{k} \})}) \to 0$ as $m \to \infty.$ Hence by Proposition \ref{prop:wasser}, $L^{(m)} \stackrel{d}{\to} L(\{ n_{k} \})$. 
On the other hand, we have already proved $L^{(m)}$ converges to $L$ in $L^2$. So $L(\{ n_{k} \})\stackrel{d}{=}L$. 

% \noindent
% The proof of mutual contiguity of $\mathbb{Q}_{n}$ and $\mathbb{P}_{n}$ now follows from Le Cam's first lemma \citet{LeCam}.
%\hfill{$\square$}

% \end{proof}

%\subsection{Proof of part 3}
% \noindent
\paragraph{Proof of \eqref{eq:janson-decomp}:}
% We now give a proof of part (b)
%The proof of this part is broadly broken into two parts. In the first part we prove $Y_{n}$ is ``close" to $$ \exp\left\{\sum_{k=1}^{K} \frac{2\mu_{k}\left(W_{n,k}\right)-\mu_{k}^{2}}{\sigma_{k}^{2}}   \right\}$$ for some large $K$. In the next part we prove that $$ \sum_{k=1}^{K} \frac{2\mu_{k}\left(W_{n,k}\right)-\mu_{k}^{2}}{\sigma_{k}^{2}} $$ is ``close" to $$\sum_{k=1}^{m_{n}} \frac{2\mu_{k}\left(W_{n,k}\right)-\mu_{k}^{2}}{\sigma_{k}^{2}}.$$ Finally by proving a result similar to continuous mapping theorem completes the proof. 
%\\
%\textbf{Step 1:}
Consider any fixed pair of $(\epsilon,\delta)\in (0,1)\times (0,1)$.
First observe that the sequence $\log(Y_{n})$ is tight from the proof of the previous part. 
% \nb{=== start here! ===}
For the given $\delta$, there exists a fixed number $M< \infty$ such that $\mathbb{P}_{n}\left[ -M \le \log(Y_{n}) \le M \right]\ge 1- \frac{1}{100}\delta$ for all $n$, 
implying $\mathbb{P}_{n}\left[  e^{-M} \le Y_{n} \le e^{M}\right]\ge 1- \frac{1}{100}\delta$. 
Now consider $\tau \in (0, e^{-M}).$ The function $\log(\cdot)$ is uniformly continuous on $[\tau, e^{M+1}]$. 
On this interval consider $\tilde{\epsilon}$ such that $\left|  \log(x)-\log(y)\right|< \frac{\epsilon}{4}$ for all $x,y$ on this interval with $|x-y|< \tilde{\epsilon}$. 
% Here $\epsilon>0$ is arbitrary.
Let $\epsilon_{1}= \min\{ \tilde{\epsilon}, e^{-M}-\tau, e^{M+1}-e^{M}\}$ and pick a sufficiently large $K\in \naturals$ such that
% Now pick $K$ large enough so that
\begin{equation}
	\label{choosingK}
% \max\left\{
\exp\left\{ \sum_{k=1}^{\infty} \frac{\mu_{k}^{2}}{\sigma_{k}^{2}}\right\}- \exp\left\{\sum_{k=1}^{K} \frac{\mu_{k}^{2}}{\sigma_{k}^{2}} 
\right\}
% \right\}
\le 
% \min\left\{
\frac{\delta \epsilon_{1}^{2}}{100}.
% \right\}.
\end{equation}

From the proof of the previous part, we also know given any subsequence $n_{l}$ there exists a further subsequence $\{ n_{l_{m}} \}$ so that under $\bbP_n$,
\begin{equation*}
\left(Y_{n_{l_{m}}}, \exp\left\{\sum_{k=1}^{K} 
\frac{2\mu_{k}\left(W_{n_{l_{m}},k}\right)-\mu_{k}^{2}}{2\sigma_{k}^{2}} \right\}\right)
\stackrel{d}{\to} 
\left(L, \exp\left\{\sum_{k=1}^{K} \frac{2\mu_{k}\left(Z_{k}\right)-\mu_{k}^{2}}{2\sigma_{k}^{2}}\right\}\right)	
\end{equation*}
and 
\begin{equation*}
\Expect\left[ \left(L- \exp\left\{\sum_{k=1}^{K} \frac{2\mu_{k}\left(Z_{k}\right)-\mu_{k}^{2}}{2\sigma_{k}^{2}}\right\}\right)^2 \right]\le \frac{\delta \epsilon_{1}^{2}}{100}.	
\end{equation*}
As a consequence,
\begin{equation}\label{eqn:bdd:liketoseries}
\begin{split}
\limsup_{n_{l_{m}} \to \infty} \mathbb{P}_{n_{l_{m}}}\left[ \left|Y_{n_{l_{m}}}- \exp\left\{\sum_{k=1}^{K} \frac{2\mu_{k}\left(W_{n_{l_{m}},k}\right)-\mu_{k}^{2}}{2\sigma_{k}^{2}}\right\}\right| \ge \frac{\epsilon_{1}}{2} \right]&\\
\le \mathbb{P}\left[ \left| (L- \exp\left\{\sum_{k=1}^{K} \frac{2\mu_{k}\left(Z_{k}\right)-\mu_{k}^{2}}{2\sigma_{k}^{2}}\right\} \right|\ge \frac{\epsilon_{1}}{2} \right]&\le \frac{\delta}{25}.
\end{split}
\end{equation}
As a consequence, for large values of $n_{l_{m}}$,
\begin{equation}\label{bdd:liketoseriesII}
\begin{split}
\mathbb{P}_{n_{l_{m}}}\left[  \left|Y_{n_{l_{m}}}- \exp\left\{\sum_{k=1}^{K} \frac{2\mu_{k}\left(W_{n_{l_{m}},k}\right)-\mu_{k}^{2}}{2\sigma_{k}^{2}} \right\}\right| 
\ge \frac{\epsilon_{1}}{2} \,\,\mbox{and}\,\, Y_{n_{l_{m}}} \notin [e^{-M},e^{M}] \right]
&\\ 
\le \frac{\delta}{25}+ \frac{\delta}{100}
& <\frac{\delta}{2}.
\end{split}
\end{equation}
Therefore, for large values of $n_{l_m}$,
\begin{equation*}
\mathbb{P}_{n_{l_{m}}}\left[ \left| \log(Y_{n_{l_{m}}}) - \left\{\sum_{k=1}^{K} \frac{2\mu_{k}\left(W_{n_{l_{m}},k}\right)-\mu_{k}^{2}}{2\sigma_{k}^{2}} \right\}\right|\ge \frac{\epsilon}{2} \right] \le  \frac{\delta}{2}.
\end{equation*}
This completes the proof.
\end{proof}